\documentclass[final,leqno,onefignum,onetabnum]{siamltex1213}
\usepackage{fancyhdr}
\usepackage{amsmath}
\usepackage{mathrsfs}
\usepackage{amsfonts}
\newtheorem{remark}{\textbf{Remark}}
\newtheorem{example}{\emph{Example}}

\title{STOCHASTIC CONFORMAL MULTI-SYMPLECTIC METHOD FOR
DAMPED STOCHASTIC NONLINEAR SCHR\"{O}DINGER
EQUATION\thanks{This work was
supported by NNSFC (NO. 91530118, NO. 11021101, NO. 11290142, NO. 11471310 and NO. 11601032).}}

\author{Chuchu Chen\footnotemark[2]\thanks{Department of Mathematics, Michigan State University.
(\email{chenc118@msu.edu}).}
\and Jialin Hong\footnotemark[3]\thanks{Institute of Computational Mathematics and Scientific/Engineering Computing, Academy of Mathematics and Systems Science, Chinese Academy of Sciences, Beijing 100190, China.
(\email{hjl@lsec.cc.ac.cn}). }\
\and Lihai Ji\footnotemark[4]\thanks{Corresponding author. Institute of Applied Physics and Computational Mathematics, Beijing 100094, China.
(\email{jilihai@lsec.cc.ac.cn}). }}

\begin{document}
\maketitle
\slugger{mms}{xxxx}{xx}{x}{x--x}

\begin{abstract}
In this paper, we propose a stochastic conformal multi-symplectic method for a class of damped stochastic Hamiltonian partial differential equations in order to inherit the intrinsic properties, and apply the numerical method to solve a kind of damped stochastic nonlinear Schr\"{o}dinger equation with multiplicative noise. It is shown that the stochastic conformal multi-symplectic method preserves the discrete stochastic conformal multi-symplectic conservation law, the discrete charge exponential dissipation law almost surely, and we also deduce the recurrence
relation of the discrete global energy. Numerical experiments are preformed to verify the good performance of the proposed stochastic conformal multi-symplectic method, compared with a Crank-Nicolson type method. Finally, we present the mean square convergence result of the proposed numerical method in temporal direction numerically.

\end{abstract}



\pagestyle{myheadings}
\thispagestyle{plain}
\markboth{CHUCHU CHEN, JIALIN HONG AND LIHAI JI}{STOCHASTIC CONFORMAL MULTI-SYMPLECTIC METHOD}

\section{Introduction}
As is well known, for deterministic Hamiltonian patial differential equations (PDEs), multi-symplectic methods were introduced by Marsden et al. \cite{RefMars}, and Bridges and Reich \cite{RefBrid}. There are now extensive research literatures concerning structure-preserving algorithms for Hamiltonian PDEs, including not only the construction of corresponding numerical methods but also the analysis on accuracy, efficiency and long-time behavior. For a detailed description of the methods as well as their implementation and applications, we refer readers to the review articles \cite{RefBrid2,RefWang} and references therein. For stochastic Hamiltonian PDEs, \cite{RefJiang} proposed stochastic multi-symplectic conservation law of stochastic Hamiltonian PDEs for the first time, and developed a stochastic multi-symplectic
method to solve stochastic nonlinear Schr\"{o}dinger (NLS) equation numerically. Recently, \cite{Refhai} derived stochastic multi-symplectic structure for three-dimensional stochastic Maxwell equations with additive noise by stochastic version of variational principle, and proposed a stochastic multi-symplectic
method that preserves the discrete stochastic multi-symplectic conservation law and stochastic energy dissipative property.

In the concluding remarks of \cite{RefBrid3}, a forced-damped nonlinear wave equation was mentioned with
a ``perturbed conservation law which takes a concise form and suggests that --- with increasing time --- average value of symplecticity is drained out of the field." In the context of Hamiltonian ordinary differential equations (ODEs) with linear damping, which are known as
conformal Hamiltonian systems \cite{Ref11}, \cite{RefMcla1} constructed numerical methods that preserve the so called conformal conservation law. Recently, \cite{RefMore,RefMore2} extended the conformal method to damped Hamiltonian PDEs and provided a framework of the construction of numerical methods which could exactly preserve the considered conformal multi-symplectic structure.

To the best of our knowledge, there is no reference about conformal multi-symplectic structure for damped stochastic Hamiltonian PDEs till now. This motivates us to investigate damped stochastic Hamiltonian PDEs with such structure, and propose numerical method which could preserve the discrete version of stochastic conformal multi-symplectic structure. To this end, we take the damped stochastic NLS equation as the keystone
mainly because it describes many physical phenomena and plays an important role in fluid dynamics, nonlinear optics, plasma physics, etc., see \cite{RefFalk,RefFalk2} and reference therein. Damping effect cannot be neglected in this case and has to be counterbalanced by amplifiers. \cite{RefGuy} presented a multilevel resolution method for the weakly damped stochastic NLS equation. This method gives better results with significantly shorter CPU time
than the other numerical methods used in the literature. In \cite{RefDebu}, a damped stochastic NLS equation driven by an additive noise was studied. And by using a coupling method, the authors established convergence of the Markov transition semi-group toward a unique invariant probability measure.

The rest of this paper is organized as follows. In section 2, we
begin with some preliminary results about damped stochastic NLS equation and show that the charge satisfies an exponential dissipation law, moreover, we present the relationship fullfilled by the energy. In section 3, we propose and analyze the stochastic conformal multi-symplectic method for the damped stochastic Hamiltonian PDEs. Section 4 is
contributed to the theoretical analysis of properties of the proposed
stochastic conformal multi-symplectic method for the damped stochastic NLS equation, including the discrete charge exponential dissipation law and the recursion formula of the discrete energy. In section 5, numerical experiments are
performed to testify the effectiveness of the stochastic conformal multi-symplectic method. And we investigate the mean-square convergence result numerically. Concluding remarks are given in Section 6. Some proofs and calculations are postponed to the final appendix.

\section{Damped stochastic nonlinear Schr\"{o}dinger equation}
In order to simplify the notations, in this paper we consider one-dimensional damped stochastic NLS equation. However, the approach and some of theoretical results can be extended to the general d-dimensional ($d\geq 2$) problem. In the case of a multiplicative noise, we consider the equation:
\begin{equation}\label{NLS}
du+(\alpha u-iu_{xx}-i|u|^{2}u)dt=i\varepsilon u\circ dW,~~t\geq0,~x\in\mathbb{R},
\end{equation}
with an initial condition
\begin{equation}\label{Initial}
u(x,0)=u_{0}(x).
\end{equation}
 Here, $u=u(x,t)$ is a complex-valued function, $W$ is a
real-valued Wiener process, $\alpha\geq 0$ is the absorption coefficient, $\varepsilon>0$ describes the size of
the noise. The $\circ$ in the last term in ({\color{blue}\ref{NLS}}) means that the product is of
Stratonovich type. Let $(\Omega,\mathcal{F},P)$ be the probability space with filtration $\{\mathcal{F}_{t}:~t\geq0\}$. Let $\{\beta_{k}:~k\in \mathbb{N}\}$ be a sequence of independent Brownian motions which are associated
with $\{\mathcal{F}_{t}:~t\geq0\}$. Let $\{e_{k}\}_{k\in \mathbb{N}}$ be an orthonormal basis of $L^{2}(\mathbb{R},\mathbb{R})$, $\phi\in\mathcal{L}_{2}(L^2,H^{\gamma})$ which is the space consisting of Hilbert-Schmidt operators from $L^{2}(\mathbb{R})$ into $H^{\gamma}(\mathbb{R})$ $(\gamma> 0)$. Then
\begin{equation}
W(t,x,\omega)=\sum_{k=0}^{\infty}\beta_{k}(t,\omega)\phi e_{k}(x),~~t\geq0,~x\in \mathbb{R},~\omega\in\Omega,
\end{equation}
is a Wiener process on the space of square integrable functions on $\mathbb{R}$, with covariance
operator $\phi\phi^{\ast}$.

We will use the equivalent It\^{o} equation. Defining the function
\begin{equation}\label{Fphi}
F_{\phi}(x)=\sum_{k=0}^{\infty}(\phi e_{k}(x))^{2},~x\in\mathbb{R}
\end{equation}
which does not depend on the basis $\{e_{k}\}_{k\in\mathbb{N}}$, the equivalent It\^{o} equation is
\begin{equation}\label{Ito}
du+(\alpha u+\frac{\varepsilon^{2}}{2}F_{\phi}u-iu_{xx}-i|u|^{2}u)dt=i\varepsilon udW.
\end{equation}

The existence and uniqueness of solutions of (\ref{NLS}) in space $H^1({\mathbb R}^d)$ were studied in \cite{RefBouard3,RefBRZ} for
more general nonlinear term case.

Moreover, equation (\ref{NLS}) possesses the charge exponential dissipation law which is an important criteria of measuring whether a
numerical simulation is good or not. It is stated in the following theorem.
\begin{theorem}\label{chargele}
The damped stochastic NLS equation (\ref{NLS}) possesses the charge exponential dissipation law almost surely
\begin{equation}\label{charge}
\int_{\mathbb{R}}|u(x,t)|^{2}dx=e^{-2\alpha t}\int_{\mathbb{R}}|u_{0}(x)|^{2}dx.
\end{equation}
\end{theorem}

\begin{proof}
To investigate the evolution relationship of the charge. We apply It\^{o} formula to the function
$$\Phi(t,u)=e^{2\alpha t}\int_{\mathbb{R}}|u|^{2}dx.$$
It holds
\begin{equation}
\begin{split}
\frac{\partial\Phi}{\partial t}&=2\alpha e^{2\alpha t}\int_{\mathbb{R}}|u|^{2}dx,\\
D\Phi(u)(\psi)&=2e^{2\alpha t}Re\Big(\int_{\mathbb{R}}\bar{u}\psi dx\Big),\\
D^{2}\Phi(u)(\psi,\varphi)&=2e^{2\alpha t}Re\Big(\int_{\mathbb{R}}\bar{\psi}\varphi dx\Big),\\
\end{split}
\end{equation}
where $D,~D^{2}$ denote the first and second Fr\'{e}chet derivative, respectively. $\bar{u}$ is the conjugate of function $u$.

It\^{o} formula leads to
\begin{equation}
\begin{split}
\Phi(t,u(t))=&\Phi(0,u_{0})+2\alpha\int_{0}^{t}\int_{\mathbb{R}}e^{2\alpha s}|u|^{2}dxds\\
&+2\int_{0}^{t}e^{2\alpha s}Re\Big(\int_{\mathbb{R}}\bar{u}\Big[-\alpha u-\frac{\varepsilon^{2}}{2}F_{\phi}u+iu_{xx}+i|u|^{2}u\Big]dx\Big)ds\\
&+2\int_{0}^{t}e^{2\alpha s}Re\Big(\int_{\mathbb{R}}i\varepsilon|u|^{2}dx\Big)dW(s)+\int_{0}^{t}e^{2\alpha s}\int_{\mathbb{R}}\varepsilon^{2}|u|^{2}F_{\phi}dxds.
\end{split}
\end{equation}
Since
\begin{equation}
\begin{split}
\int_{0}^{t}e^{2\alpha s}\int_{\mathbb{R}}&\Big(\alpha|u(s)|^{2}+\frac{\varepsilon^{2}}{2}F_{\phi}|u(s)|^{2}\Big)dxds\\
&=\alpha\int_{0}^{t}e^{2\alpha s}\int_{\mathbb{R}}|u(s)|^{2}dxds+\int_{0}^{t}e^{2\alpha s}\int_{\mathbb{R}}\frac{\varepsilon^{2}}{2}F_{\phi}|u(s)|^{2}dxds.
\end{split}
\end{equation}
Combing all of these equations, we obtain
\begin{equation*}
\Phi(t,u(t))=\Phi(0,u_{0}),
\end{equation*}
which means
\begin{equation*}
\int_{\mathbb{R}}|u(x,t)|^{2}dx=e^{-2\alpha t}\int_{\mathbb{R}}|u_{0}(x)|^{2}dx.
\end{equation*}
Thus the proof is finished.
\end{proof}

\begin{remark}
Note that the above Theorem states that the $L^2({\mathbb R})$-norm of the solution of \eqref{NLS}
 decays exponentially with exponent $\alpha$.
If $\alpha=0$, the charge exponential dissipation law (\ref{charge}) becomes
\begin{equation*}
\int_{\mathbb{R}}|u(x,t)|^{2}dx=\int_{\mathbb{R}}|u_{0}(x)|^{2}dx.
\end{equation*}
This is consistent with the result of the Proposition 4.4 in \cite{RefBouard3}.
\end{remark}

The Hamiltonian plays an important role in the study of the nonlinear Schr\"odinger equation. It is a conserved quantity in
the absence of noise and damping. It is defined by
\begin{equation}
H(u)=\frac{1}{2}\int_{\mathbb{R}}|\nabla u|^{2}dx-\frac{1}{4}\int_{\mathbb{R}}|u|^{4}dx.
\end{equation}
We now state the following energy evolution law for damped stochastic NLS equation (\ref{NLS}).

\begin{theorem}
Let $u_0\in H^1({\mathbb R})$ and $\phi\in {\mathcal L}_2(L^2,H^1)$.
The damped stochastic NLS equation (\ref{NLS}) has the following global energy evolution law almost surely,
\begin{equation}\label{energy}
\begin{split}
H(u(t))=e^{-2\alpha t}H(u_{0})&+\frac{\alpha}{2}\int_{0}^{t}\int_{\mathbb{R}}e^{-2\alpha(t-s)}|u(s)|^{4}dxds\\
&-\int_{0}^{t}Im\Big(\int_{\mathbb{R}}\varepsilon e^{-2\alpha(t-s)}\bar{u}\nabla udx\Big)\nabla dW(s)\\
&+\frac{\varepsilon^{2}}{2}\sum_{k=0}^{\infty}\int_{0}^{t}\int_{\mathbb{R}}e^{-2\alpha(t-s)}|u(s)|^{2}|\nabla \phi e_{k}|^{2}dxds.
\end{split}
\end{equation}
\end{theorem}

\begin{proof}
The proof is similar to that of Theorem \ref{chargele} by applying It\^o formula to functional $H(u)$.  Since $H(u)$ is Fr\'{e}chet derivable, the derivatives of $H(u)$ along directions $\psi$ and $(\psi,\phi)$ are as follows:
\begin{equation*}
\begin{split}
D H(u)(\psi)&=Re\int_{\mathbb{R}}\nabla\bar{u}\nabla\psi dx-Re\int_{\mathbb{R}}|u|^{2}\bar{u}\psi dx,\\
D^{2}H(\psi,\varphi)&=Re\int_{\mathbb{R}}\nabla\bar{\psi}\nabla\varphi dx-Re\int_{\mathbb{R}}|u|^{2}\bar{\psi}\varphi dx-2\int_{\mathbb{R}}Re(\bar{u}\psi)Re(\bar{u}\varphi)dx.
\end{split}
\end{equation*}
From It\^{o} formula, we have
\begin{equation}\label{hamil}
\begin{split}
d H(u)=&DH(u)(du)+\frac12 D^2H(u)(du,du)\\
=&-2\alpha H(u)dt+\frac{\alpha}{2}\int_{\mathbb R}|u|^4dx dt-\varepsilon Im \int_{\mathbb R} \bar{u}\nabla u dx d(\nabla W(t))\\
&+\frac{\varepsilon^2}{2}\int_{\mathbb R}|u|^2\sum_{k\in {\mathbb N}}\big(\nabla(\phi e_{k})\big)^2dxdt,
\end{split}
\end{equation}
which leads to
\begin{equation*}
\begin{split}
H(u(t))=&e^{-2\alpha t}H(u_{0})+\frac{\alpha}{2}\int_{0}^{t}\int_{\mathbb{R}}e^{-2\alpha(t-s)}|u(s)|^{4}dxds\\
&-\int_{0}^{t}Im\int_{\mathbb{R}}\varepsilon e^{-2\alpha(t-s)}u\nabla \bar{u}\nabla dW(s)dx\\
&+\frac{\varepsilon^{2}}{2}\sum_{k=0}^{\infty}\int_{0}^{t}\int_{\mathbb{R}}e^{-2\alpha(t-s)}|u(s)|^{2}|\nabla\phi e_{k}|^{2}dxds.
\end{split}
\end{equation*}

\end{proof}

\begin{remark}
If $\alpha=0$, the energy evolution law (\ref{energy}) becomes
\begin{equation*}
\begin{split}
H(u(t))=H(u_{0})&-\int_{0}^{t}Im\Big(\int_{\mathbb{R}}\varepsilon u\nabla \bar{u}dx\Big)\nabla dW(s)\\
&+\frac{\varepsilon^{2}}{2}\sum_{k=0}^{\infty}\int_{0}^{t}\int_{\mathbb{R}}|u(s)|^{2}|\nabla \phi e_{k}|^{2}dxds.
\end{split}
\end{equation*}
This is consistent with the result of the Proposition 4.5 in \cite{RefBouard3}.
\end{remark}

Furthermore, we have the following estimate for the Hamiltonian.
\begin{theorem}\label{est_Hamil}
  Let $u_0\in H^1({\mathbb R})$ and $\phi\in {\mathcal L}_{2}(L^2,H^2)$. Then there exists a positive constant $C$ such that
  \begin{equation}
  {\mathbb E}[H(u(t))]\leq e^{-\alpha t}{\mathbb E}[H(u_0)]+C.
  \end{equation}
\end{theorem}
\begin{proof}
  We start from equality \eqref{hamil}, i.e.,
  \begin{equation*}
    dH(u)+2\alpha H(u)dt=\frac{\alpha}{2}\int_{\mathbb R}|u|^4dx dt
+\frac{\varepsilon^2}{2}\int_{\mathbb R}|u|^2\sum_{k\in {\mathbb N}}\big(\nabla(\phi e_{k})\big)^2dxdt+dM(t),
  \end{equation*}
  where $dM(t)=-\varepsilon Im \int_{\mathbb R} \bar{u}\nabla u dx d(\nabla W(t))$.

  Note that we use Gagliardo-Nirenberg inequality to have
  \[
  \int_{\mathbb R}|u|^4dx\leq 2\|\nabla u\|_{L^2({\mathbb R})}^{2}+C\|u\|_{L^2({\mathbb R})}^{6}\leq  2\|\nabla u\|_{L^2({\mathbb R})}^{2}+C_1
  \]
  and use the embedding of $L^{\infty}({\mathbb R})$ into $ H^1({\mathbb R})$ to have
  \[
  \frac{\varepsilon^2}{2}\int_{\mathbb R}|u|^2\sum_{k\in {\mathbb N}}\big(\nabla(\phi e_{k})\big)^2dx
  \leq \frac{\varepsilon^2}{2}\|u\|_{L^2({\mathbb R})}^2\|\phi\|_{{\mathcal L}_2(L^2,H^2)}^2\leq C_2,
  \]
  where we utilize Theorem \ref{chargele}.

  Thus we have
  \[
  dH(u)+\alpha H(u)dt\leq dM(t)+(C_1+C_2)dt,
  \]
  i.e.,
  \[
  d{\mathbb E}[H(u)]+\alpha {\mathbb E}[H(u)]dt\leq (C_1+C_2)dt.
  \]
  By multiplying $e^{\alpha t}$ to both sides of the above inequality and then taking integral from $0$ to $t$, we finish the proof.
\end{proof}

In order to give an estimate of the $H^1({\mathbb R})$-norm of the solution of equation \eqref{NLS},
 we utilize again the  Gagliardo-Nirenberg inequality
\[
\|u\|_{L^4({\mathbb R})}^{4}\leq \|\nabla u\|_{L^2(\mathbb R)}^{2}+C\|u\|_{L^2({\mathbb R})}^{6},\quad \forall u\in H^{1}({\mathbb R})
\]
to have
\[
\frac12 \|\nabla u\|_{L^2({\mathbb R})}^2=H(u)+\frac14 \|u\|_{L^4({\mathbb R})}^{4}
\leq H(u)+\frac14 \|\nabla u\|_{L^2(\mathbb R)}^{2}+\frac14 C\|u\|_{L^2({\mathbb R})}^{6},
\]
which means that there exists a positive constant $C$ such that
\[
\|\nabla u\|_{L^2({\mathbb R})}^2\leq 4 H(u)+C.
\]
Hence we know that under Theorem \ref{est_Hamil}, the $H^1({\mathbb R})$-norm of the solution of equation \eqref{NLS} has a uniform bound.

\section{Stochastic conformal multi-symplecitc method}
In this section, we propose a stochastic conformal multi-symplectic method for damped stochastic Hamiltonian PDEs. Finally, two examples are given to illustrate how to construct the Hamiltonian structure.

Consider the following damped stochastic Hamiltonian PDEs:
\begin{equation}\label{SFDHPDE}
Mz_{t}+Kz_{x}=\nabla S_{1}(z)+\nabla S_{2}(z)\circ \dot{\chi}+Dz,~~z\in\mathbb{R}^{d},
\end{equation}
where, $M$ and $K$ are skew-symmetric matrices, $D=-\frac{a}{2}M-\frac{b}{2}K$, and $\dot{\chi}=\frac{dW(t)}{dt}$ is a real-valued white noise which is
delta correlated in time, and either smooth or delta correlated in space. The gradients
of $S_{1}$ and $S_{2}$ are with respect to $z$

According to the mathematical definition of the noise $\dot{\chi}dt=d_{t}W$, the damped stochastic Hamiltonian PDEs (\ref{SFDHPDE}) can be rewritten into the form:
\begin{equation}
Md_{t}z+Kz_{x}dt=\nabla S_{1}(z)dt+\nabla S_{2}(z)\circ d_{t}W+Dzdt,~~z\in\mathbb{R}^{d}.
\end{equation}

We have the following theorem. As the proof is postponed to Appendix A.
\begin{theorem}\label{conlemma}
  System \eqref{SFDHPDE} preserves the stochastic conformal conservation law
  \begin{equation}\label{conformal conservation law}
    d_{t}\omega(t,x)+\partial_{x}\kappa(t,x)dt=(-a\omega(t,x)-b\kappa(t,x))dt,
  \end{equation}
  which means the following integral equality,
  \begin{align}\label{conformal conservation law_integral}
    \int_{x_{0}}^{x_{1}}\omega(t_{1},x)dx+\int_{t_{0}}^{t_{1}}\kappa(t,x_{1})dt-\int_{x_{0}}^{x_{1}}\omega(t_{0},x)dx
    -\int_{t_{0}}^{t_{1}}\kappa(t,x_{0})dt\nonumber\\
    =-\int_{x_{0}}^{x_{1}}\int_{t_{0}}^{t_{1}}a\omega(t,x)dtdx
    -\int_{x_{0}}^{x_{1}}\int_{t_{0}}^{t_{1}}b\kappa(t,x)dtdx.
  \end{align}
  Here $\omega=dz\wedge Mdz$ and $\kappa=dz\wedge Kdz$ are the differential 2-forms associated
with the two skew-symmetric matrices $M$ and $K$, respectively, and $(t_0, t_1)\times(x_0,x_1)$ is
the local definition domain of $z(t,x)$.
\end{theorem}

A naturally question is what kind of numerical methods has the ability of preserving
the discrete form of the stochastic conformal multi-symplectic conservation law when they are
applied to the damped stochastic Hamiltonian PDEs? If a numerical method can preserve the discrete stochastic conformal multi-symplectic conservation law, we call it stochastic conformal multi-symplectic method in this paper.

\subsection{Stochastic conformal multi-symplectic method}
In order to construct the stochastic conformal multi-symplectic method, we
introduce a uniform grid $(t_{n},x_{j})\in\mathbb{R}^{2}$ with mesh length $\Delta t$ in the temporal direction
and mesh length $\Delta x$ in the spatial direction, respectively. The value of the function $z(t,x)$ at the
mesh point $(t_{n},x_{j})$ is denoted by $z_{j}^{n}$. In addition, define the difference operators
\begin{equation}\label{df}
\delta_{t}^{\frac{a}{2}}z_{j}^{n}=\frac{z_{j}^{n+1}-e^{-\frac{a}{2}\Delta t}z_{j}^{n}}{\Delta t},~~\delta_{x}^{\frac{b}{2}}z_{j}^{n}=\frac{z_{j+1}^{n}-e^{-\frac{b}{2}\Delta x}z_{j}^{n}}{\Delta x},
\end{equation}
and the averaging operators
\begin{equation}\label{av}
A_{t}^{\frac{a}{2}}z_{j}^{n}=\frac{z_{j}^{n+1}+e^{-\frac{a}{2}\Delta t}z_{j}^{n}}{2},~~A_{x}^{\frac{b}{2}}z_{j}^{n}=\frac{z_{j+1}^{n}+e^{-\frac{b}{2}\Delta x}z_{j}^{n}}{2}.
\end{equation}

These two operators have the following two properties.

\begin{lemma}\cite{RefMore2}\label{lemma1}
The operators (\ref{df}) and (\ref{av}) commute, i.e.,
\begin{equation}
\begin{split}
\delta_{\xi}^{\alpha}A_{\eta}^{\beta}z_{j}^{n}&=A_{\eta}^{\beta}\delta_{\xi}^{\alpha}z_{j}^{n},\\
\delta_{\xi}^{\alpha}\delta_{\eta}^{\beta}z_{j}^{n}&=\delta_{\eta}^{\beta}\delta_{\xi}^{\alpha}z_{j}^{n},\\
A_{\xi}^{\alpha}A_{\eta}^{\beta}z_{j}^{n}&=A_{\eta}^{\beta}A_{\xi}^{\alpha}z_{j}^{n}.
\end{split}
\end{equation}
\end{lemma}

\begin{lemma}\cite{RefMore2}\label{lemma2}
The operators (\ref{df}) and (\ref{av}) satisfy a discrete product rule
\begin{equation}
\delta_{\zeta}^{\alpha}\langle\phi,\psi\rangle=\Big<\delta_{\zeta}^{\frac{\alpha}{2}}\phi,A_{\zeta}^{\frac{\alpha}{2}}\psi\Big>+
  \Big<A_{\zeta}^{\frac{\alpha}{2}}\phi,\delta_{\zeta}^{\frac{\alpha}{2}}\psi\Big>.
  \end{equation}
\end{lemma}

Now, we consider the following full-discrete form
\begin{equation}\label{method}
  M(\delta_{t}^{\frac{a}{2}}A_{x}^{\frac{b}{2}}z_{j}^{n})+K(\delta_{x}^{\frac{b}{2}}A_{t}^{\frac{a}{2}}z_{j}^{n})=\nabla S_{1}(A_{t}^{\frac{a}{2}}A_{x}^{\frac{b}{2}}z_{j}^{n})+\nabla S_{2}(A_{t}^{\frac{a}{2}}A_{x}^{\frac{b}{2}}z_{j}^{n})\dot{\chi}_{j}^{n}.
\end{equation}

Here,
\begin{equation*}
\dot{\chi}_{j}^{n}=\frac{W(t_{n+1},x_{j})-W(t_{n},x_{j})}{\Delta t},
\end{equation*}
with
\begin{equation}\label{WW}
W(t_{n+1},x_{j})-W(t_{n},x_{j})=\sum_{m=1}^{\infty}\sqrt{\eta_{m}}e_{m}(x_{j})\Big(\beta_{m}(t_{n+1})-\beta_{m}(t_{n})\Big).
\end{equation}
For the full-discrete method (\ref{method}), we have the following result.
\begin{theorem}
  Discretization \eqref{method} preserves the discrete stochastic conformal multi-symplectic conservation law almost surely
  \begin{equation}
    \delta_{t}^{a}\Big<MA_{x}^{\frac{b}{2}}U_{j}^{n},A_{x}^{\frac{b}{2}}V_{j}^{n}\Big>+\delta_{x}^{b}\Big<KA_{t}^{\frac{a}{2}}U_{j}^{n},A_{t}^{\frac{a}{2}}V_{j}^{n}\Big>=0,
  \end{equation}
  where $U_{j}^{n}$ and $V_{j}^{n}$ are any solutions of the variational equation of \eqref{method}.
\end{theorem}
\begin{proof}
Lemma \ref{lemma1} and Lemma \ref{lemma2} imply
  \begin{align}
    \delta_{t}^{a}\Big<MA_{x}^{\frac{b}{2}}U_{j}^{n},A_{x}^{\frac{b}{2}}V_{j}^{n}\Big>&=\Big<M\delta_{t}^{\frac{a}{2}}A_{x}^{\frac{b}{2}}U_{j}^{n},A_{t}^{\frac{a}{2}}A_{x}^{\frac{b}{2}}V_{j}^{n}\Big>
    +\Big<MA_{t}^{\frac{a}{2}}A_{x}^{\frac{b}{2}}U_{j}^{n},\delta_{t}^{\frac{a}{2}}A_{x}^{\frac{b}{2}}V_{j}^{n}\Big>\nonumber\\
   & =\Big<M\delta_{t}^{\frac{a}{2}}A_{x}^{\frac{b}{2}}U_{j}^{n},A_{t}^{\frac{a}{2}}A_{x}^{\frac{b}{2}}V_{j}^{n}\Big>-\Big<M\delta_{t}^{\frac{a}{2}}A_{x}^{\frac{b}{2}}V_{j}^{n},A_{t}^{\frac{a}{2}}A_{x}^{\frac{b}{2}}U_{j}^{n}\Big>,
  \end{align}
  and
  \begin{align}
    \delta_{x}^{b}\Big<KA_{t}^{\frac{a}{2}}U_{j}^{n},A_{t}^{\frac{a}{2}}V_{j}^{n}\Big>=\Big<K\delta_{x}^{\frac{b}{2}}A_{t}^{\frac{a}{2}}U_{j}^{n},A_{x}^{\frac{b}{2}}A_{t}^{\frac{a}{2}}V_{j}^{n}\Big>
    -\Big<K\delta_{x}^{\frac{b}{2}}A_{t}^{\frac{a}{2}}V_{j}^{n},A_{x}^{\frac{b}{2}}A_{t}^{\frac{a}{2}}U_{j}^{n}\Big>.
  \end{align}
  Since $U_j^n$ and $V_j^n$ are any solutions of the variational equation for \eqref{method}, we know that they satisfy the following equation
  \begin{equation}
     M(\delta_{t}^{\frac{a}{2}}A_{x}^{\frac{b}{2}}dz_{j}^{n})+K(\delta_{x}^{\frac{b}{2}}A_{t}^{\frac{a}{2}}dz_{j}^{n})= (S_{1})_{zz}(A_{t}^{\frac{a}{2}}A_{x}^{\frac{b}{2}}dz_{j}^{n})+(S_{2})_{zz}(A_{t}^{\frac{a}{2}}A_{x}^{\frac{b}{2}}dz_{j}^{n})\dot{\chi}_{j}^{n}.
  \end{equation}
  Then
  \begin{equation}
  \begin{split}
    &\delta_{t}^{a}\Big<MA_{x}^{\frac{b}{2}}U_{j}^{n},A_{x}^{\frac{b}{2}}V_{j}^{n}\Big>+\delta_{x}^{b}\Big<KA_{t}^{\frac{a}{2}}U_{j}^{n},A_{t}^{\frac{a}{2}}V_{j}^{n}\Big>\\
   & =\Big<(S_{1})_{zz}(A_{t}^{\frac{a}{2}}A_{x}^{\frac{b}{2}}U_{j}^{n})+(S_{2})_{zz}(A_{t}^{\frac{a}{2}}A_{x}^{\frac{b}{2}}U_{j}^{n})\dot{\chi}_{j}^{n},
   A_{t}^{\frac{a}{2}}A_{x}^{\frac{b}{2}}V_{j}^{n}\Big>\\
   &\quad\quad-\Big<(S_{1})_{zz}(A_{t}^{\frac{a}{2}}A_{x}^{\frac{b}{2}}V_{j}^{n})+(S_{2})_{zz}(A_{t}^{\frac{a}{2}}A_{x}^{\frac{b}{2}}V_{j}^{n})\dot{\chi}_{j}^{n},
   A_{t}^{\frac{a}{2}}A_{x}^{\frac{b}{2}}U_{j}^{n}\Big>\\
   &=0,
   \end{split}
  \end{equation}
  where the last equality is due to the symmetry of $(S_1)_{zz}$ and $(S_2)_{zz}$.

  This completes the proof.
\end{proof}
\subsection{Examples}
In fact, a large class of damped stochastic PDEs can be represented as (\ref{SFDHPDE}). In this subsection, we give two concrete examples.

\begin{example}[Damped stochastic KdV equation]
Consider the following damped Stochastic KdV equation with additive noise \cite{RefGord}:
\begin{equation}\label{KdV}
u_{t}+6uu_{x}+u_{xxx}=\alpha u+\gamma\dot{\chi},~\mathrm{in}~ U\times \mathbb{R}^{+}
\end{equation}
where $U\subset\mathbb{R}$ is a bounded open set with a smooth boundary $\partial U$, $\alpha> 0$ is the damping coefficient, $\gamma$ represents the amplitude of the noise source. The coefficient 6 is presented as a matter of convenience and historical significance.

By expanding (\ref{KdV}) as a first-order system of PDEs,
\begin{equation*}
\begin{split}
u_x&=v,~~\phi_{x}=u,\\
\frac{1}{2}u_{t}+\omega_{x}&=\frac{\alpha}{2}u+\gamma\dot{\chi},\\
-\frac{1}{2}\phi_{t}-v_{x}&=3u^{2}-\omega-\frac{\alpha}{2}\phi.
\end{split}
\end{equation*}

Then, we have state variable $z = (\phi,u,v,\omega)^{T}$, Hamiltonian
\begin{equation*}
S_{1}(z)=u^{3}-u\omega+\frac{1}{2}v^{2},~~S_{2}(z)=\gamma \phi
\end{equation*}
and the pair of skew-symmetric matrices $M$, $K$ and $D$,
\begin{equation*}
M=\left(
\begin{array}{cccc}
0&\frac{1}{2}&0&0\\
-\frac{1}{2}&0&0&0\\
0&0&0&0\\
0&0&0&0
\end{array}\right),~K=\left(
\begin{array}{cccc}
0&0&0&1\\
0&0&-1&0\\
0&1&0&0\\
-1&0&0&0
\end{array}\right),~D=\left(
\begin{array}{cccc}
0&\frac{\alpha}{2}&0&0\\
-\frac{\alpha}{2}&0&0&0\\
0&0&0&0\\
0&0&0&0
\end{array}\right),
\end{equation*}
with $a=-\alpha,~b=0$.
\end{example}

\begin{example}[Damped stochastic NLS equation] Let $u(t,x)=p(t,x)+iq(t,x)$, where $p,~q$ are real-valued functions, equation (\ref{NLS}) can be separated into
\begin{equation}
\begin{split}
dp+\Big(\alpha p+q_{xx}+q(p^{2}+q^{2})\Big)dt&=-\varepsilon q\circ dW(t),\\
dq+\Big(\alpha q-p_{xx}-p(p^{2}+q^{2})\Big)dt&=\varepsilon p\circ dW(t).
\end{split}
\end{equation}

By introducing two additional new variables, $v=p_x,~\omega=q_x$, and defining a state
variable $z=(p,q,v,\omega)^{T}$ , the equation above can be transformed to the compact form
\begin{equation}
Mz_{t}+Kz_{x}=\nabla S_{1}(z)+\nabla S_{2}(z)\circ \dot{\chi}+Dz,~~z\in\mathbb{R}^{4},
\end{equation}
where
\begin{equation}\label{KL}
M=\left(
\begin{array}{cccc}
0&-1&0&0\\
1&0&0&0\\
0&0&0&0\\
0&0&0&0
\end{array}\right),~K=\left(
\begin{array}{cccc}
0&0&1&0\\
0&0&0&1\\
-1&0&0&0\\
0&-1&0&0
\end{array}\right),~D=\left(
\begin{array}{cccc}
0&\alpha&0&0\\
-\alpha&0&0&0\\
0&0&0&0\\
0&0&0&0
\end{array}\right),
\end{equation}
with $a=2\alpha,~b=0$. And
\begin{equation}\label{SS}
S_{1}(z)=-\frac{1}{2}(v^{2}+\omega^{2})-\frac{1}{4}(p^{2}+q^{2})^{2},~S_{2}(z)=-\frac{\varepsilon}{2}(p^{2}+q^{2}).
\end{equation}
\end{example}

In the following section, we consider the damped stochastic NLS equation (\ref{NLS}) to illustrate the merits of stochastic conformal multi-symplectic method.

\section{Conservative properties of the stochastic conformal multi-symplectic method}
This section investigates the global conservative properties of the stochastic conformal multi-symplectic method (\ref{method}) for the damped stochastic NLS equation.

For convenience, we substitute matrices (\ref{KL}) and Hamiltonian (\ref{SS}) into (\ref{method}), and rewrite the numerical method componentwise
\begin{equation}
\begin{split}
\delta_{t}^{\alpha}A_{x}^{0}p_{j}^{n}+\delta_{x}^{0}A_{t}^{\alpha}\omega_{j}^{n}&=-\Big((A_{t}^{\alpha}A_{x}^{0}p_{j}^{n})^{2}+(A_{t}^{\alpha}A_{x}^{0}q_{j}^{n})^{2}\Big)A_{t}^{\alpha}A_{x}^{0}q_{j}^{n}-\varepsilon A_{t}^{\alpha}A_{x}^{0}q_{j}^{n}\dot{\chi}_{j}^{n},\\
\delta_{t}^{\alpha}A_{x}^{0}q_{j}^{n}-\delta_{x}^{0}A_{t}^{\alpha}v_{j}^{n}&=~~\Big((A_{t}^{\alpha}A_{x}^{0}p_{j}^{n})^{2}+(A_{t}^{\alpha}A_{x}^{0}q_{j}^{n})^{2}\Big)A_{t}^{\alpha}A_{x}^{0}p_{j}^{n}+\varepsilon A_{t}^{\alpha}A_{x}^{0}p_{j}^{n}\dot{\chi}_{j}^{n},\\
\delta_{x}^{0}A_{t}^{\alpha}p_{j}^{n}&=~~A_{t}^{\alpha}A_{x}^{0}v_{j}^{n},\\
\delta_{x}^{0}A_{t}^{\alpha}q_{j}^{n}&=~~A_{t}^{\alpha}A_{x}^{0}\omega_{j}^{n}.
\end{split}
\end{equation}

Recalling that $u=p+iq$, and eliminating the additionally introduced variables $v$ and
$\omega$, we get the equation of $u$:
\begin{equation}\label{sn}
\begin{split}
\Big(\delta_{t}^{\alpha}A_{x}^{0}u_{j+1}^{n}+\delta_{t}^{\alpha}A_{x}^{0}u_{j}^{n}\Big)&-2i(\delta_{x}^{0})^{2}A_{t}^{\alpha}u_{j}^{n}\\
&=i\Big(|A_{t}^{\alpha}A_{x}^{0}u_{j}^{n}|^{2}A_{t}^{\alpha}A_{x}^{0}u_{j}^{n}+|A_{t}^{\alpha}A_{x}^{0}u_{j+1}^{n}|^{2}A_{t}^{\alpha}A_{x}^{0}u_{j+1}^{n}\Big)\\
&+i\varepsilon\Big(A_{t}^{\alpha}A_{x}^{0}u_{j}^{n}\dot{\chi}_{j}^{n}+A_{t}^{\alpha}A_{x}^{0}u_{j+1}^{n}\dot{\chi}_{j+1}^{n}\Big).
\end{split}
\end{equation}

The following discussions are all based on the full-discretized stochastic conformal multi-symplectic method (\ref{sn}).
\subsection{Equivalent form}
In this subsection, we consider the relation between stochastic multi-symplectic and stochastic conformal multi-symplectic.

Consider the damped stochastic NLS equation (\ref{NLS}). Let $\varpi=e^{\alpha t}u(t,x)$, then It\^{o} formula leads to
\begin{equation}\label{tran}
\begin{split}
d\varpi&=\alpha e^{\alpha t}udt+e^{\alpha t}du\\
&=(i\varpi_{xx}+ie^{-2\alpha t}|\varpi|^{2}\varpi)dt+i\varepsilon \varpi\circ dW.
\end{split}
\end{equation}
Note that equation \eqref{tran} is a nonlinear Schr\"odinger equation with varying coefficient, and we refer readers to
\cite{RefHL, RefHLMZ} for the studies of multi-symplectic schemes of Schr\"odinger equation with varying coefficients in the absence of noise.

Set $\varpi(t,x)=r(t,x)+is(t,x)$, where $r$ and $s$ are real-valued functions, and introduce two new variables $\xi=r_{x},~\eta=s_{x}$, we can rewrite equation (\ref{tran}) as
\begin{equation}\label{shs}
\begin{split}
dr&=-\Big(\eta_{x}+e^{-2\alpha t}(r^{2}+s^{2})s\Big)dt-\varepsilon s\circ dW(t),\\
ds&=\Big(\xi_{x}+e^{-2\alpha t}(r^{2}+s^{2})r\Big)dt+\varepsilon r\circ dW(t),\\
\xi&=r_{x},\\
\eta&=s_{x}.
\end{split}
\end{equation}

Let $U=(r,s,\xi,\eta)^{T}$, then equation (\ref{tran}) can be rewritten as the following stochastic Hamiltonian PDEs
\begin{equation}
\tilde{M}U_{t}+\tilde{K}U_{x}=\nabla S_{1}(t,U)+\nabla S_{2}(U)\circ\dot{\chi},
\end{equation}
where
\begin{equation*}
\tilde{M}=\left(
\begin{array}{cccc}
0&1&0&0\\
-1&0&0&0\\
0&0&0&0\\
0&0&0&0
\end{array}\right),~\tilde{K}=\left(
\begin{array}{cccc}
0&0&-1&0\\
0&0&0&-1\\
1&0&0&0\\
0&1&0&0
\end{array}\right),
\end{equation*}
and
\begin{equation*}
S_{1}(t,U)=\frac{1}{2}(\xi^{2}+\eta^{2})+\frac{1}{4}e^{-2\alpha t}(r^{2}+s^{2})^{2},~S_{2}(U)=\frac{\varepsilon}{2}(r^{2}+s^{2}).
\end{equation*}

We consider the following full-discrete scheme
\begin{equation}\label{dis}
\tilde{M}\delta_{t}^{0}U_{j+\frac{1}{2}}^{n}+\tilde{K}\delta_{x}^{0}U_{j}^{n+\frac{1}{2}}=\nabla S_{1}(t_{n+\theta},U_{j+\frac{1}{2}}^{n+\frac{1}{2}})+\nabla S_{2}(U_{j+\frac{1}{2}}^{n+\frac{1}{2}})\dot{\chi}_{j}^{n},
\end{equation}
where $t_{n+\theta}=t_{n}+\theta\Delta t$ with $\theta\in[0,\;1]$
and
\begin{equation*}
\begin{split}
U_{j}^{n+\frac{1}{2}}&=\frac{1}{2}(U_{j}^{n+1}+U_{j}^{n}),~U_{j+\frac{1}{2}}^{n}=\frac{1}{2}(U_{j+1}^{n}+U_{j}^{n}),\\
U_{j+\frac{1}{2}}^{n+\frac{1}{2}}&=\frac{1}{4}(U_{j}^{n+1}+U_{j}^{n+1}+U_{j+1}^{n}+U_{j}^{n}).
\end{split}
\end{equation*}

In the absence of noise, and $\theta=\frac12$, scheme \eqref{dis} is the central box scheme proposed in \cite{RefHL, RefHLMZ} to discretize Schr\"odinger
equation with variable coefficents by multi-symplectic method. In the stochastic context,
similar as in \cite{RefJiang}, one may show that the discretization (\ref{dis}) is a stochastic multi-symplectic method. Moreover, we will see that the stochastic conformal multi-symplectic method (\ref{sn}) is equivalent to this stochastic multi-symplectic method in the case of $\theta=1$.

Equations (\ref{dis}) can be rewritten as
\begin{equation}
\begin{split}
-\delta_{t}^{0}r_{j+\frac{1}{2}}^{n}-\delta_{x}^{0}\eta_{j}^{n+\frac{1}{2}}&=e^{-2\alpha t_{n+\theta}}\Big((r_{j+\frac{1}{2}}^{n+\frac{1}{2}})^{2}+(s_{j+\frac{1}{2}}^{n+\frac{1}{2}})^{2}\Big)s_{j+\frac{1}{2}}^{n+\frac{1}{2}}+\varepsilon s_{j+\frac{1}{2}}^{n+\frac{1}{2}}\dot{\chi}_{j}^{n},\\[1mm]
\delta_{t}^{0}s_{j+\frac{1}{2}}^{n}-\delta_{x}^{0}\xi_{j}^{n+\frac{1}{2}}&=e^{-2\alpha t_{n+\theta}}\Big((r_{j+\frac{1}{2}}^{n+\frac{1}{2}})^{2}+(s_{j+\frac{1}{2}}^{n+\frac{1}{2}})^{2}\Big)r_{j+\frac{1}{2}}^{n+\frac{1}{2}}+\varepsilon r_{j+\frac{1}{2}}^{n+\frac{1}{2}}\dot{\chi}_{j}^{n},\\[1mm]
\delta_{x}^{0}r_{j}^{n+\frac{1}{2}}&=\xi_{j+\frac{1}{2}}^{n+\frac{1}{2}},\\[1mm]
\delta_{x}^{0}s_{j}^{n+\frac{1}{2}}&=\eta_{j+\frac{1}{2}}^{n+\frac{1}{2}}.
\end{split}
\end{equation}

Recalling that $\varpi=r+is$, and eliminating the additionally introduced variables $\xi$ and
$\eta$, we get the equation of $\varpi$:
\begin{equation}\label{sms}
\begin{split}
\Big(\delta_{t}^{0}\varpi_{j+\frac{1}{2}}^{n}+\delta_{t}^{0}\varpi_{j-\frac{1}{2}}^{n}\Big)&-2i(\delta_{x}^{0})^{2}\varpi_{j-1}^{n+\frac{1}{2}}\\
&=ie^{-2\alpha t_{n+\theta}}\Big(|\varpi_{j+\frac{1}{2}}^{n+\frac{1}{2}}|^{2}\varpi_{j+\frac{1}{2}}^{n+\frac{1}{2}}+|\varpi_{j-\frac{1}{2}}^{n+\frac{1}{2}}|^{2}\varpi_{j-\frac{1}{2}}^{n+\frac{1}{2}}\Big)\\
&+i\varepsilon\Big(\varpi_{j+\frac{1}{2}}^{n+\frac{1}{2}}\dot{\chi}_{j}^{n}+\varpi_{j-\frac{1}{2}}^{n+\frac{1}{2}}\dot{\chi}_{j-1}^{n}\Big).
\end{split}
\end{equation}

Then, we have the following result.
\begin{theorem}
The stochastic conformal multi-symplectic method (\ref{sn}) is equivalent to the stochastic multi-symplectic method (\ref{sms}) with $\theta=1$.
\end{theorem}
\begin{proof}
It follows from the transform $\varpi=e^{\alpha t}u$, we set
\begin{equation}\label{trans}
\begin{split}
&\varpi_{j}^{n}=e^{\alpha t_{n}}u_{j}^{n},\\
&\varpi_{j+\frac{1}{2}}^{n}=e^{\alpha t_{n}}A_{x}^{0}u_{j}^{n},\\
&\varpi_{j}^{n+\frac{1}{2}}=\frac{1}{2}\Big(e^{\alpha t_{n+1}}u_{j}^{n+1}+e^{\alpha t_{n}}u_{j}^{n}\Big)=e^{\alpha t_{n+1}}A_{t}^{\alpha}u_{j}^{n}.
\end{split}
\end{equation}

Substitute (\ref{trans}) into equation (\ref{sms}), we can obtain
\begin{equation}\label{equs}
\begin{split}
\delta_{t}^{0}&\Big(e^{\alpha t_{n}}u_{j+\frac{1}{2}}^{n}+e^{\alpha t_{n}}u_{j-\frac{1}{2}}^{n}\Big)-2i(\delta_{x}^{0})^{2}e^{\alpha t_{n+1}}A_{t}^{\alpha}u_{j-1}^{n}\\
&=ie^{-2\alpha t_{n+\theta}}\Big[\Big|e^{\alpha t_{n+1}}A_{x}^{0}A_{t}^{\alpha}u_{j}^{n}\Big|^{2}e^{\alpha t_{n+1}}A_{x}^{0}A_{t}^{\alpha}u_{j}^{n}
+\Big|e^{\alpha t_{n+1}}A_{x}^{0}A_{t}^{\alpha}u_{j-1}^{n}\Big|^{2}e^{\alpha t_{n+1}}A_{x}^{0}A_{t}^{\alpha}u_{j-1}^{n}\Big]\\
&\quad+i\varepsilon\Big[e^{\alpha t_{n+1}}A_{x}^{0}A_{t}^{\alpha}u_{j}^{n}\dot{\chi}_{j}^{n}+e^{\alpha t_{n+1}}A_{x}^{0}A_{t}^{\alpha}u_{j-1}^{n}\dot{\chi}_{j-1}^{n}\Big].
\end{split}
\end{equation}

Multiply equation (\ref{equs}) by $e^{-\alpha t_{n+1}}$, and
noting that
\[
\delta_{t}^{0}\varpi_{j}^{n}=\delta_{t}^{0}\big(e^{\alpha t_{n}}u_{j}^{n}\big)=e^{\alpha t_{n+1}}\delta_{t}^{\alpha}u_{j}^{n},
\]
 we can see the equivalence of this two numerical methods if $\theta=1$. Thus the proof is finished.
\end{proof}

\subsection{Properties}
In this subsection, we present two properties of stochastic conformal multi-symplectic method (\ref{sn}) for damped stochastic NLS equation.
\begin{theorem}\label{charge_dis}
The stochastic conformal multi-symplectic
method (\ref{sn}) has the discrete charge exponential dissipation law almost surely,
\begin{equation}
\sum_{j}|u_{j+\frac{1}{2}}^{n+1}|^{2}=e^{-2\alpha \Delta t}\sum_{j}|u_{j+\frac{1}{2}}^{n}|^{2},
\end{equation}
with $u_{j+\frac{1}{2}}^{n}=A_{x}^{0}u_{j}^{n}=\frac{1}{2}(u_{j+1}^{n}+u_{j}^{n})$.
\end{theorem}

\begin{proof}
Multiply equation (\ref{sn}) by $A_{t}^{\alpha}\bar{u}_{j+1}^{n}$, i.e., the conjugate of $A_{t}^{\alpha}u_{j+1}^{n}$, sum over all
spatial grid points $j$, and take the real part. Then, the first term of left-side becomes
\begin{equation*}
\begin{split}
Re&\Big(\sum_{j}(\delta_{t}^{\alpha}A_{x}^{0}u_{j}^{n}+\delta_{t}^{\alpha}A_{x}^{0}u_{j+1}^{n})A_{t}^{\alpha}\bar{u}_{j+1}^{n}\Big)\\
&=\frac{1}{4\Delta t}Re\sum_{j}(u_{j+1}^{n+1}-e^{-\alpha \Delta t}u_{j+1}^{n}+u_{j}^{n+1}-e^{-\alpha \Delta t}u_{j}^{n})(\bar{u}_{j+1}^{n+1}+e^{-\alpha \Delta t}\bar{u}_{j+1}^{n})\\
&+\frac{1}{4\Delta t}Re\sum_{j}(u_{j+2}^{n+1}-e^{-\alpha \Delta t}u_{j+2}^{n}+u_{j+1}^{n+1}-e^{-\alpha \Delta t}u_{j+1}^{n})(\bar{u}_{j+1}^{n+1}+e^{-\alpha \Delta t}\bar{u}_{j+1}^{n})\\
&=\frac{1}{\Delta t}\Big(\sum_{j}|u_{j+\frac{1}{2}}^{n+1}|^{2}-e^{-2\alpha\Delta t}\sum_{j}|u_{j+\frac{1}{2}}^{n}|^{2}\Big).
\end{split}
\end{equation*}
The second term of left-side becomes
\begin{equation*}
\begin{split}
&Re\Big(\sum_{j}(-2i(\delta_{x}^{0})^{2}A_{t}^{\alpha}u_{j}^{n})A_{t}^{\alpha}\bar{u}_{j+1}^{n}\Big)\\
&=\frac{1}{2\Delta x^{2}}Im\sum_{j}(u_{j+2}^{n+1}-2u_{j+1}^{n+1}+u_{j}^{n+1}+e^{-\alpha \Delta t}(u_{j+2}^{n}-2u_{j+1}^{n}+u_{j}^{n}))(\bar{u}_{j+1}^{n+1}+e^{-\alpha \Delta t}\bar{u}_{j+1}^{n})\\
&=0.
\end{split}
\end{equation*}
Similarly, for the first term of right-side of (\ref{sn}) it holds
\begin{equation*}
\begin{split}
&Re\Big(\sum_{j}i\Big(|A_{t}^{\alpha}A_{x}^{0}u_{j}^{n}|^{2}A_{t}^{\alpha}A_{x}^{0}u_{j}^{n}+|A_{t}^{\alpha}A_{x}^{0}u_{j+1}^{n}|^{2}A_{t}^{\alpha}A_{x}^{0}u_{j+1}^{n}\Big)A_{t}^{\alpha}\bar{u}_{j+1}^{n}\Big)\\
&=\frac{-1}{8}\sum_{j}|A_{t}^{\alpha}A_{x}^{0}u_{j}^{n}|^{2}Im(u_{j+1}^{n+1}+e^{-\alpha \Delta t}u_{j+1}^{n}+u_{j}^{n+1}+e^{-\alpha \Delta t}u_{j}^{n}))(\bar{u}_{j+1}^{n+1}+e^{-\alpha \Delta t}\bar{u}_{j+1}^{n})\\
&-\frac{1}{8}\sum_{j}|A_{t}^{\alpha}A_{x}^{0}u_{j+1}^{n}|^{2}Im(u_{j+2}^{n+1}+e^{-\alpha \Delta t}u_{j+2}^{n}+u_{j+1}^{n+1}+e^{-\alpha \Delta t}u_{j+1}^{n}))(\bar{u}_{j+1}^{n+1}+e^{-\alpha \Delta t}\bar{u}_{j+1}^{n})\\
&=0.
\end{split}
\end{equation*}
In the similar way, the real part of the last term of right-side of (\ref{sn}) vanishes since $\dot{\chi}_{j}^{n}$ is real-valued. Combining all these equalities, we obtain the discrete charge conservation law
\begin{equation*}
\frac{1}{\Delta t}\Big(\sum_{j}|u_{j+\frac{1}{2}}^{n+1}|^{2}-e^{-2\alpha\Delta t}\sum_{j}|u_{j+\frac{1}{2}}^{n}|^{2}\Big)=0.
\end{equation*}
Thus, the proof is finished.
\end{proof}

The result of this theorem is evidently consistent with the continuous version of charge dissipation law
(\ref{charge}), which means that the charge exponential dissipation law can be exactly preserved by the
proposed stochastic conformal multi-symplectic method.

The next result concerns the discrete global energy evolution relationship of the damped stochastic NLS equation.
\begin{theorem}
The stochastic conformal multi-symplectic
method (\ref{sn}) satisfies the following recursion of discrete global energy conservation
law almost surely,
\begin{equation}
\begin{split}
\sum_{j}|\delta_{x}^{0}u_{j}^{n+1}|^{2}&-\sum_{j}|A_{t}^{\alpha}A_{x}^{0}u_{j}^{n}|^{2}|A_{x}^{0}u_{j}^{n+1}|^{2}\\
=&e^{-2\alpha \Delta t}\Big(\sum_{j}|\delta_{x}^{0}u_{j}^{n}|^{2}-\sum_{j}|A_{t}^{\alpha}A_{x}^{0}u_{j}^{n}|^{2}|A_{x}^{0}u_{j}^{n}|^{2}\Big)\\
&+\varepsilon\Big(\sum_{j}|A_{x}^{0}u_{j}^{n+1}|^{2}-e^{-2\alpha \Delta t}\sum_{j}|A_{x}^{0}u_{j}^{n}|^{2}\Big)\dot{\chi}_{j}^{n}.
\end{split}
\end{equation}
\end{theorem}
\begin{proof}
Multiplying equation (\ref{sn}) by $\delta_{t}^{\alpha}\bar{u}_{j+1}^{n}$, summing up for $j$ over the spatial
domain, and taking the imaginary part, we obtain the results as follows.

The first term of left-side reads
\begin{equation}
\begin{split}
&Im\Big(\sum_{j}(\delta_{t}^{\alpha}A_{x}^{0}u_{j}^{n}+\delta_{t}^{\alpha}A_{x}^{0}u_{j+1}^{n})\delta_{t}^{\alpha}\bar{u}_{j+1}^{n}\Big)\\
&=\frac{1}{2\Delta t^{2}}Im\sum_{j}(u_{j+1}^{n+1}-e^{-\alpha \Delta t}u_{j+1}^{n}+u_{j}^{n+1}-e^{-\alpha \Delta t}u_{j}^{n})(\bar{u}_{j+1}^{n+1}-e^{-\alpha \Delta t}\bar{u}_{j+1}^{n})\\
&+\frac{1}{2\Delta t^{2}}Im\sum_{j}(u_{j+2}^{n+1}-e^{-\alpha \Delta t}u_{j+2}^{n}+u_{j+1}^{n+1}-e^{-\alpha \Delta t}u_{j+1}^{n})(\bar{u}_{j+1}^{n+1}-e^{-\alpha \Delta t}\bar{u}_{j+1}^{n})\\
&=0.
\end{split}
\end{equation}
Similarly, for the second term of left-side of (\ref{sn}) it holds
\begin{equation}
\begin{split}
&Im\Big(\sum_{j}(-2i(\delta_{x}^{0})^{2}A_{t}^{\alpha}u_{j}^{n})\delta_{t}^{\alpha}\bar{u}_{j+1}^{n}\Big)\\
&=\frac{-1}{\Delta x^{2}\Delta t}Re\sum_{j}(u_{j+2}^{n+1}-2u_{j+1}^{n+1}+u_{j}^{n+1}+e^{-\alpha \Delta t}(u_{j+2}^{n}-2u_{j+1}^{n}+u_{j}^{n}))(\bar{u}_{j+1}^{n+1}-e^{-\alpha \Delta t}\bar{u}_{j+1}^{n})\\
&=\frac{1}{\Delta x^{2}\Delta t}\sum_{j}\Big(|u_{j+1}^{n+1}-u_{j}^{n+1}|^{2}-e^{-2\alpha \Delta t}|u_{j+1}^{n}-u_{j}^{n}|^{2}\Big).
\end{split}
\end{equation}
For the first term of right-side of (\ref{sn}), we have
\begin{equation}
\begin{split}
&Im\Big(\sum_{j}i\Big(|A_{t}^{\alpha}A_{x}^{0}u_{j}^{n}|^{2}A_{t}^{\alpha}A_{x}^{0}u_{j}^{n}+|A_{t}^{\alpha}A_{x}^{0}u_{j+1}^{n}|^{2}A_{t}^{\alpha}A_{x}^{0}u_{j+1}^{n}\Big)\delta_{t}^{\alpha}\bar{u}_{j+1}^{n}\Big)\\
&=\frac{1}{4\Delta t}\sum_{j}|A_{t}^{\alpha}A_{x}^{0}u_{j}^{n}|^{2}Re(u_{j+1}^{n+1}+e^{-\alpha \Delta t}u_{j+1}^{n}+u_{j}^{n+1}+e^{-\alpha \Delta t}u_{j}^{n}))(\bar{u}_{j+1}^{n+1}-e^{-\alpha \Delta t}\bar{u}_{j+1}^{n})\\
&+\frac{1}{4\Delta t}\sum_{j}|A_{t}^{\alpha}A_{x}^{0}u_{j+1}^{n}|^{2}Re(u_{j+2}^{n+1}+e^{-\alpha \Delta t}u_{j+2}^{n}+u_{j+1}^{n+1}+e^{-\alpha \Delta t}u_{j+1}^{n}))(\bar{u}_{j+1}^{n+1}-e^{-\alpha \Delta t}\bar{u}_{j+1}^{n})\\
&=\frac{1}{\Delta t}\sum_{j}|A_{t}^{\alpha}A_{x}^{0}u_{j}^{n}|^{2}\Big(|A_{x}^{0}u_{j}^{n+1}|^{2}-e^{-2\alpha \Delta t}|A_{x}^{0}u_{j}^{n}|^{2}\Big).
\end{split}
\end{equation}
For the last term of right-hand of (\ref{sn}), it leads
\begin{equation}
\begin{split}
Im\Big(\sum_{j}i\varepsilon&\Big(A_{t}^{\alpha}A_{x}^{0}u_{j}^{n}\circ\dot{\chi}_{j}^{n}+A_{t}^{\alpha}A_{x}^{0}u_{j+1}^{n}\circ\dot{\chi}_{j+1}^{n}\Big)\delta_{t}^{\alpha}\bar{u}_{j+1}^{n}\Big)\\
&=\frac{\varepsilon}{\Delta t}\Big(|A_{x}^{0}u_{j}^{n+1}|^{2}-e^{-2\alpha\Delta t}|A_{x}^{0}u_{j}^{n}|^{2}\Big)\dot{\chi}_{j}^{n}.
\end{split}
\end{equation}
Combining all these equations, finally, we get
\begin{equation}
\begin{split}
\frac{1}{\Delta t}\sum_{j}|\delta_{x}^{0}u_{j}^{n+1}|^{2}&-\frac{1}{\Delta t}\sum_{j}|A_{t}^{\alpha}A_{x}^{0}u_{j}^{n}|^{2}|A_{x}^{0}u_{j}^{n+1}|^{2}\\
=&\frac{1}{\Delta t}e^{-2\alpha \Delta t}\Big(\sum_{j}|\delta_{x}^{0}u_{j}^{n}|^{2}-\sum_{j}|A_{t}^{\alpha}A_{x}u_{j}^{n}|^{2}|A_{x}^{0}u_{j}^{n}|^{2}\Big)\\
&+\frac{\varepsilon}{\Delta t}\Big(\sum_{j}|A_{x}^{0}u_{j}^{n+1}|^{2}-e^{-2\alpha \Delta t}\sum_{j}|A_{x}^{0}u_{j}^{n}|^{2}\Big)\dot{\chi}_{j}^{n}.
\end{split}
\end{equation}
Thus, the proof is finished.
\end{proof}

\section{Numerical experiments}
In this section we provide three numerical examples to illustrate the accuracy and capability of the method developed in the previous sections. We investigate the good performance of the stochastic conformal multi-symplectic method, compared with  a Crank-Nicolson type method which is non conformal multi-symplectic. Furthermore, we check the temporal accuracy by
fixing the space step sufficiently small such that errors stemming
from the spatial approximation are negligible.

 In the following, we take the spatial domain as $x\in[x_{L},\;x_{R}]$ and boundary conditions as
 \begin{equation}\label{bc}
 u(x_L,t)=u(x_R,t)=0,
 \end{equation}
 and use the mesh
 \begin{align*}
& x_{j}=x_{L}+j\Delta x,\,j=1,2,\cdots,J:=\lfloor\frac{x_{R}-x_{L}}{\Delta x}\rfloor,\\
& t_{n}=n\Delta t,\, n=1,2,\cdots
 \end{align*}
 for our numerical computations.

 In each sub-interval $[t_n,\;t_{n+1}]$, under the initial condition $u_0(x)$ and boundary condition \eqref{bc}, we write \eqref{sn} as the form
 \[
 A(n)U^{n+1}=B(n)U^{n}+F(t_n,t_{n+1},U^n,U^{n+1},\Delta W^{n+1}),\,n=1,2,\cdots,
 \]
 where $A(n)$, $B(n)$ are invertible tridiagonal matrices depending on coefficients of the  equation, the vector
 $U^n=(u_{1}^{n},u_2^n,\cdots,u_J^n)^{T}$ and $F$ denotes the discretization of nonlinear and noise terms.

\vspace{3mm}
\noindent {\it Example 1.} Consider the damped stochastic NLS equation \eqref{NLS} with
 $W(t)$ being a standard Brownian motion, and in this case its
 exact plane wave solution is given by
\begin{equation}\label{pws}
  u(t)=Ae^{-\alpha t}e^{i\left(\frac{1}{2\alpha}|A|^2-\frac{e^{-2\alpha t}}{2\alpha}|A|^2+\varepsilon W(t)\right)}.
\end{equation}
We compare the proposed
stochastic conformal multi-symplectic method \eqref{method} with the exact relationship.
Let the spatial domain $[x_{L},\, x_{R}]$ be $[0,2\pi]$, $A=0.5$, $\alpha=0.1$ and $\varepsilon=\sqrt{2}$. Fig. \ref{amp} plots the exact amplitude
with the numerical values, the left one is exact and numerical values for amplitude averaged
over 1000 trajectories at time $T=5$; while the right one is the error in amplitudes of stochastic conformal multi-symplectic method  compared with exact solution, which is computed by the spatial averages. We observe that the error in amplitudes is of $10^{-14}$ scale, thus the proposed stochastic conformal multi-symplectic method preserves the amplitude well.
\begin{figure}[th!]
\begin{center}
  \includegraphics[height=4.1cm,width=5.5cm]{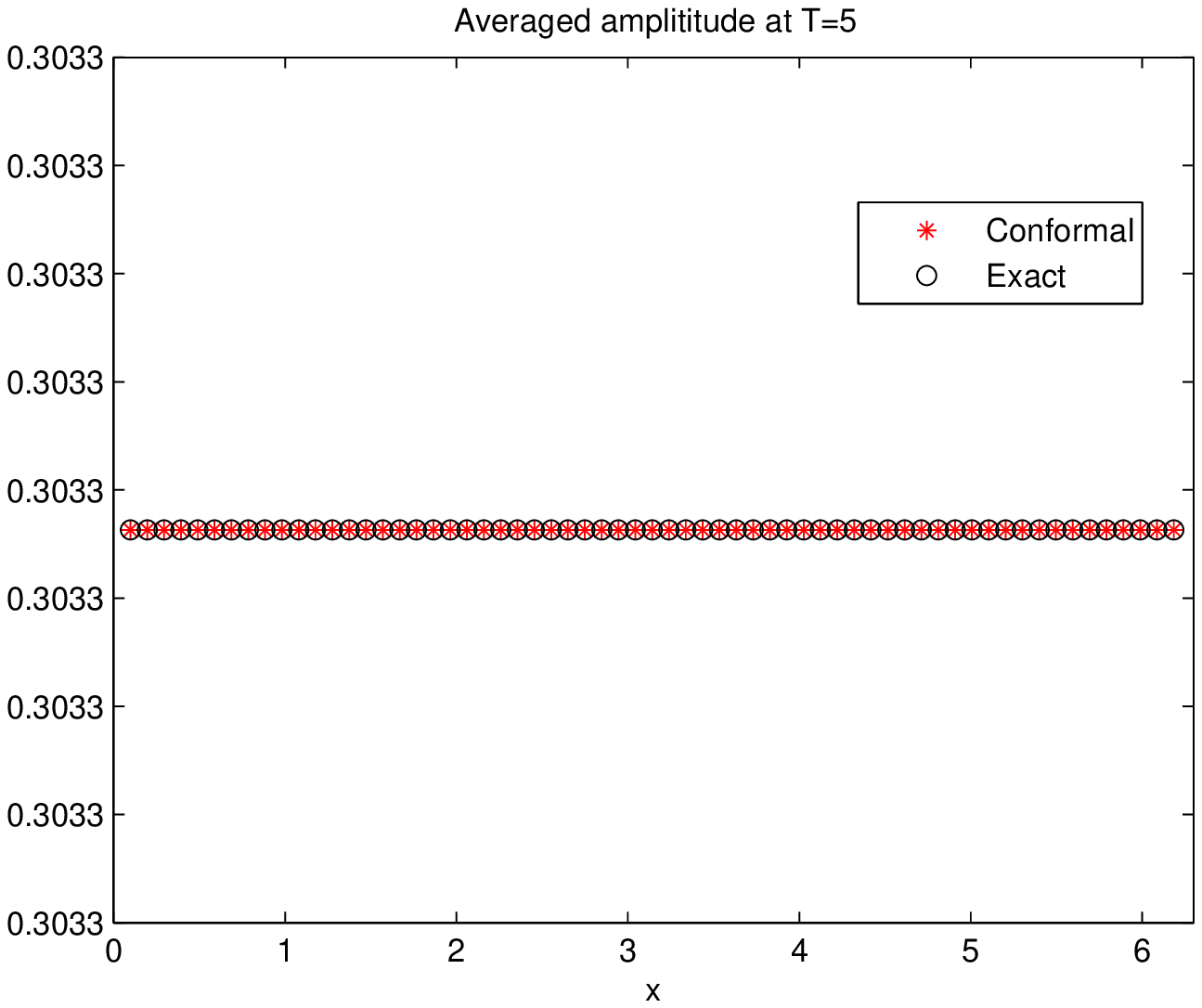}
  \includegraphics[height=4.1cm,width=5.5cm]{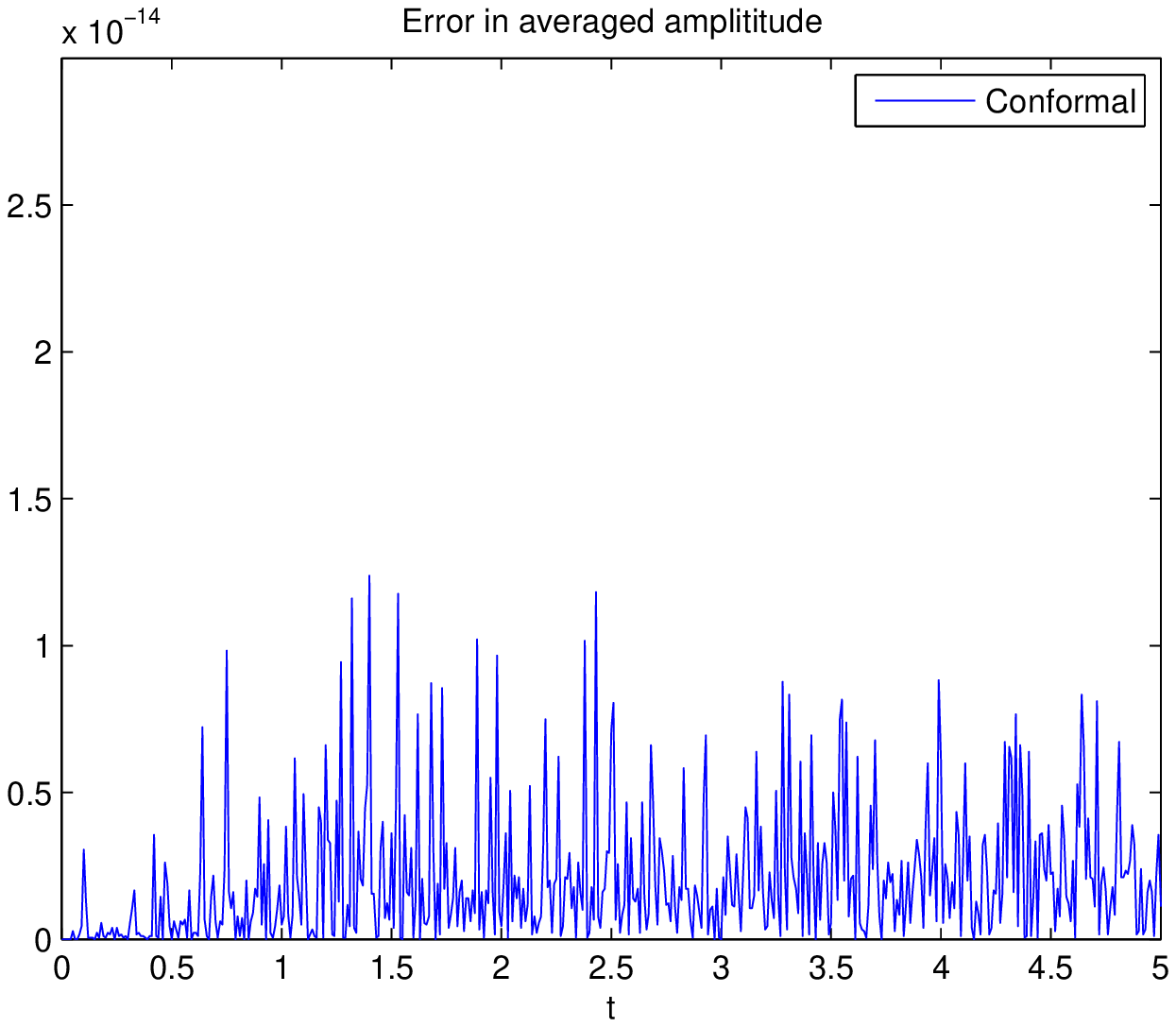}
  \caption{Averaged amplitude of the conformal multi-symplectic method for the plane wave solution \eqref{pws} with $T=5$, $\Delta t=0.01$ over 1000 paths.}\label{amp}
  \end{center}
\end{figure}

Fig. \ref{phase} plots the exact averaged phase
with the numerical values.
The left one is exact and numerical values for phase averaged
over $10^{4}$ trajectories at time $T=5$; while the right one is the error in phase of conformal multi-symplectic method
compared with exact solution, which is computed by the spatial averages. We observe that the error in phase is significant, which means that  the proposed stochastic conformal multi-symplectic method may alter the wave speeds.
\begin{figure}[th!]
\begin{center}
  \includegraphics[height=4.1cm,width=5.5cm]{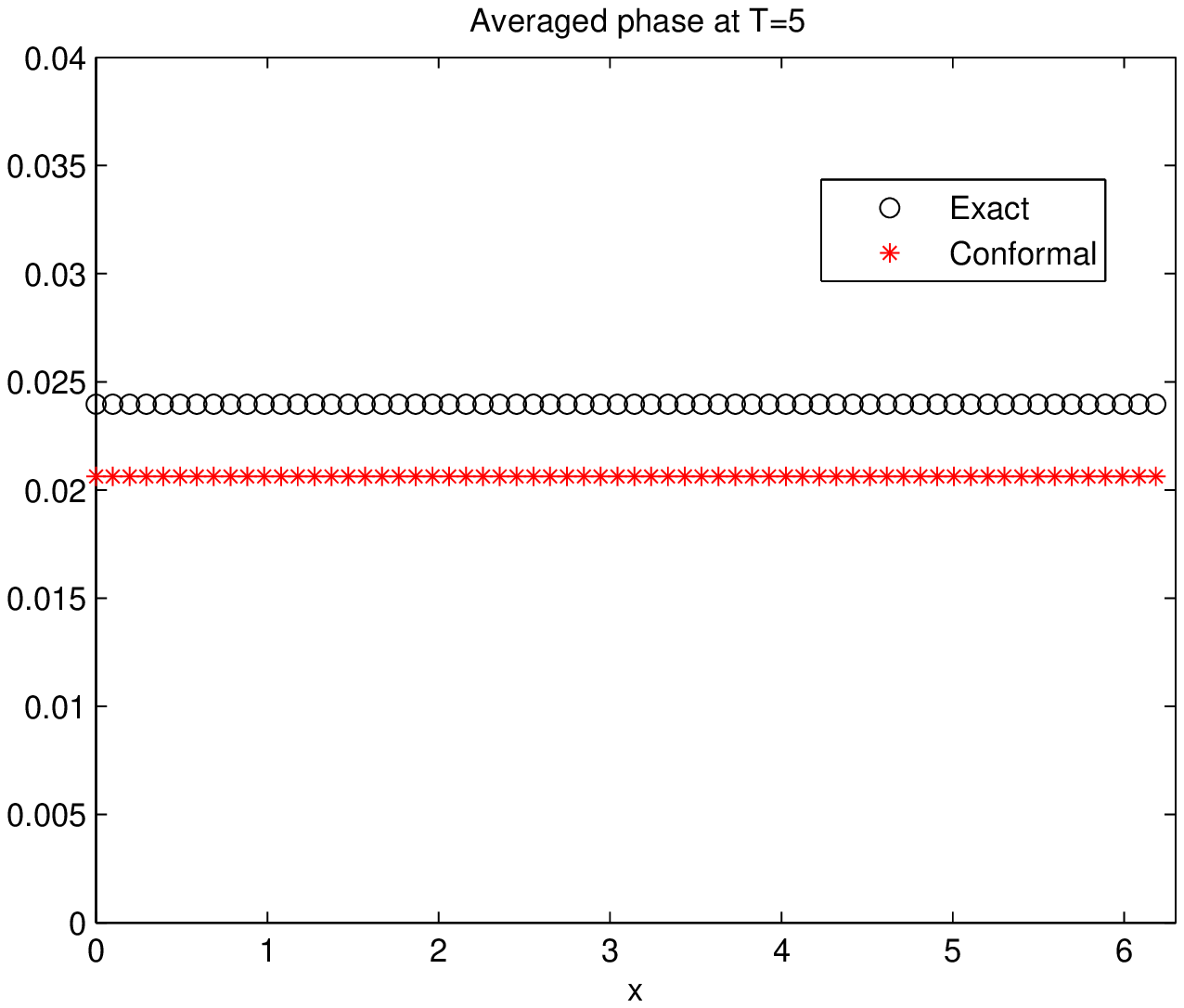}
  \includegraphics[height=4.1cm,width=5.5cm]{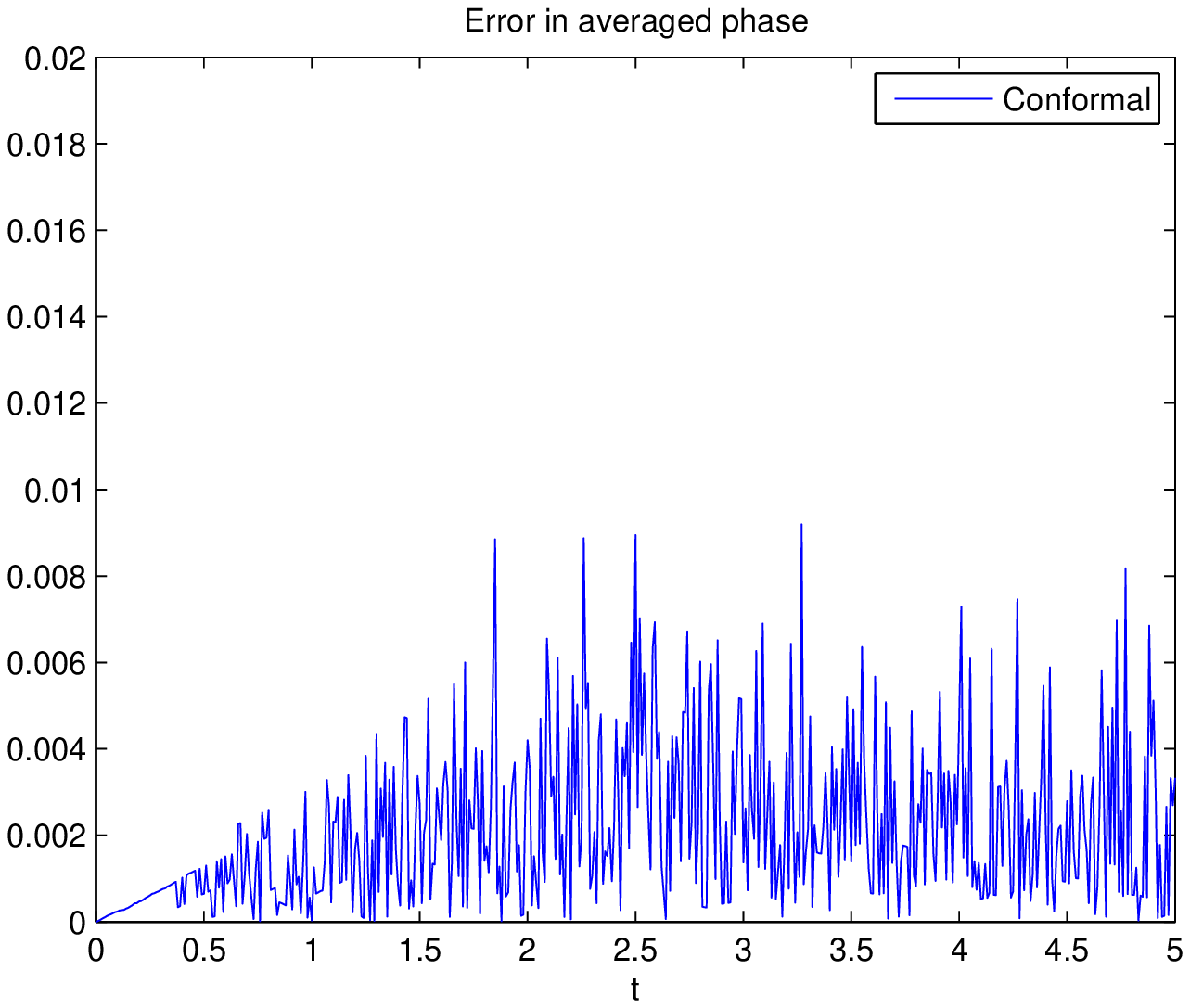}
  \caption{Averaged phase of the conformal multi-symplectic method for the plane wave solution \eqref{pws} with $T=5$, $\Delta t=0.01$ over $10^{4}$ paths.}\label{phase}
  \end{center}
\end{figure}

\vspace{3mm}
\noindent{\it Example 2.}
The next example is performed for the case that the noise depends on not only time $t$ but also on space variable $x$, whose exact solution is difficult to derived.
Here, the spatial domain $[x_{L},x_{R}]$ is $[-25,25]$, $\varepsilon=0.5$, and the initial value is given by $u|_{t=0}={\rm sech}(x)$. For each numerical experiment in this example, we take the spatial meshgrid-size $\Delta x=0.1$, and the longest time interval $[0,10]$. Furthermore, we take $\eta_{m}=1$ and the orthnormal basis $e_{m}(x)=\sqrt{\frac{2}{X_{R}-X_{L}}}\sin\Big(\frac{m\pi (x-X_{L})}{X_{R}-X_{L}}\Big)$ in equation \eqref{WW}. And we truncate the infinite series of real-valued Wiener process \eqref{WW} till $M=8$.


To compare the stochastic conformal multi-symplectic method in terms of solution behavior, we construct the following Crank-Nicolson type numerical scheme:
\begin{equation}\label{non_confor}
\delta_{t}^{0}u_{j}^{n}+\alpha u_{j}^{n+\frac{1}{2}}-i\delta_{x}^{0}\delta_{x}^{0}u_{j-1}^{n+\frac{1}{2}}-iu_{j}^{n+\frac{1}{2}}A_{t}^{0}|u_{j}^{n}|^{2}=i\varepsilon u_{j}^{n+\frac{1}{2}}\dot{\chi}_{j}^{n}.
\end{equation}
It is obvious that this method is neither stochastic conformal multi-symplectic nor multi-symplectic.
By multiplying \eqref{non_confor} with $\bar{u}_{j}^{n+\frac{1}{2}}$ which is the conjugate of $u_{j}^{n+\frac{1}{2}}$, taking the real part and then summing over all spatial grid
points $j$, we get the following charge dissipation law
\begin{equation}
  \sum_{j}|u_{j}^{n+1}|^{2}=\sum_{j}|u_{j}^{n}|^{2}-2\alpha\sum_{j}|u_{j}^{n+\frac{1}{2}}|^{2},
\end{equation}
from which we may also see some kind of dissipation relation of charge. To investigate this property further, we take
 $e^{-2\alpha t}\mathcal{Q}(0)$ as the standard criterion with $\mathcal{Q}(0)$ denoting the initial charge, since   we know that from Theorem \ref{chargele} it is the relationship satisfied by the charge at time $t$ in the continuous problem.
 Fig. \ref{charge00} plots the discrete averaged charge with the exact one, the left ones are exact and two numerical  charge evolution relationships for different values of $\alpha$; while the right ones are the residual of charge: $2\alpha\Delta t-\log\Big(\frac{Q^{n}}{Q^{n+1}}\Big)$, with $Q^{n}$ being discrete charge of stochastic conformal multi-symplectic and Crank-Nicolson methods respectively. We may observe that the stochastic conformal multi-symplectic method provides better fits for the dissipation rate than the non-conformal method. And as the growth of $\alpha$, the residual of Crank-Nicolson scheme becomes larger; while the residual of stochastic conformal multi-symplectic method remains zero.

\begin{figure}[th!]
\begin{center}
  \includegraphics[height=4.0cm,width=4.5cm]{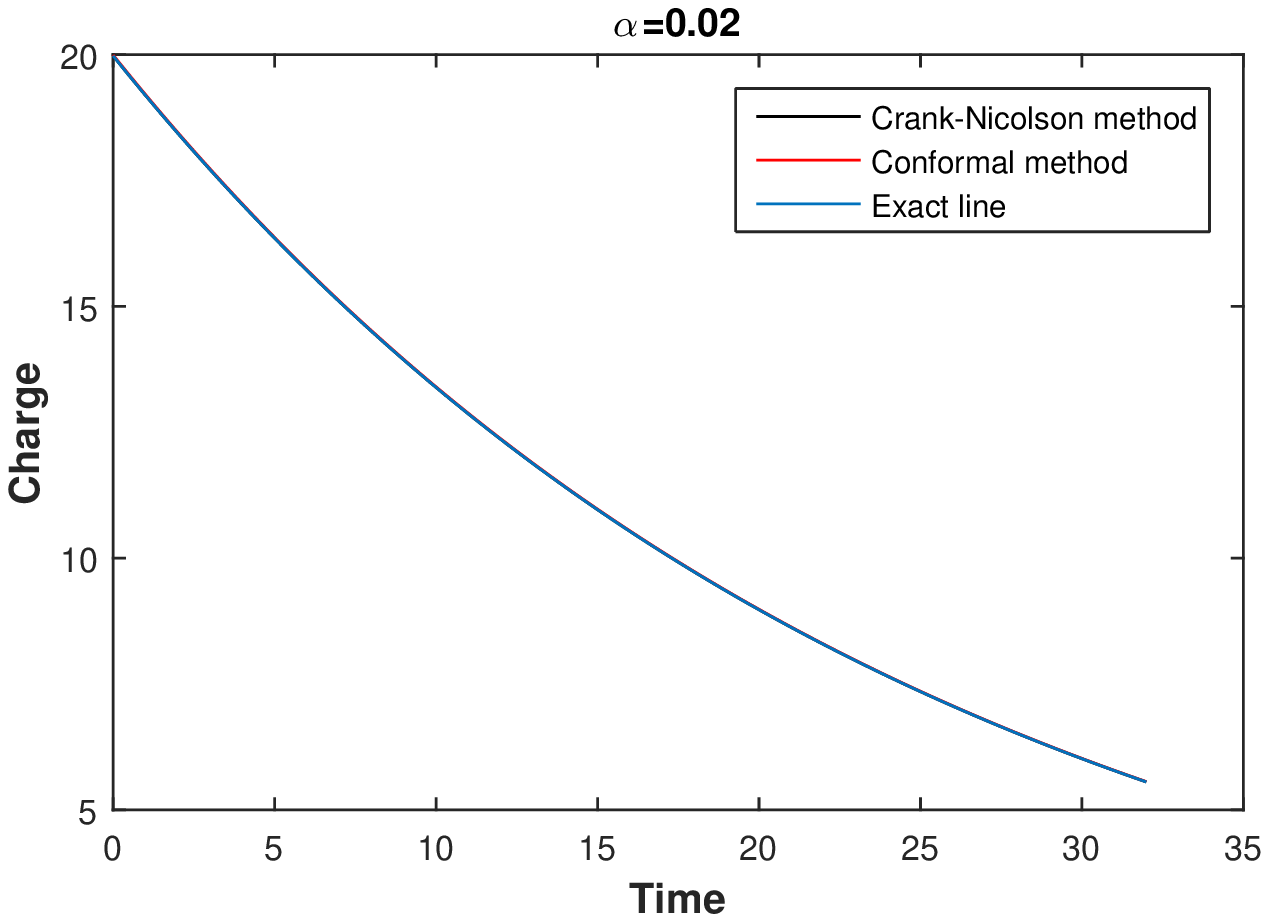}
  \includegraphics[height=4.0cm,width=4.5cm]{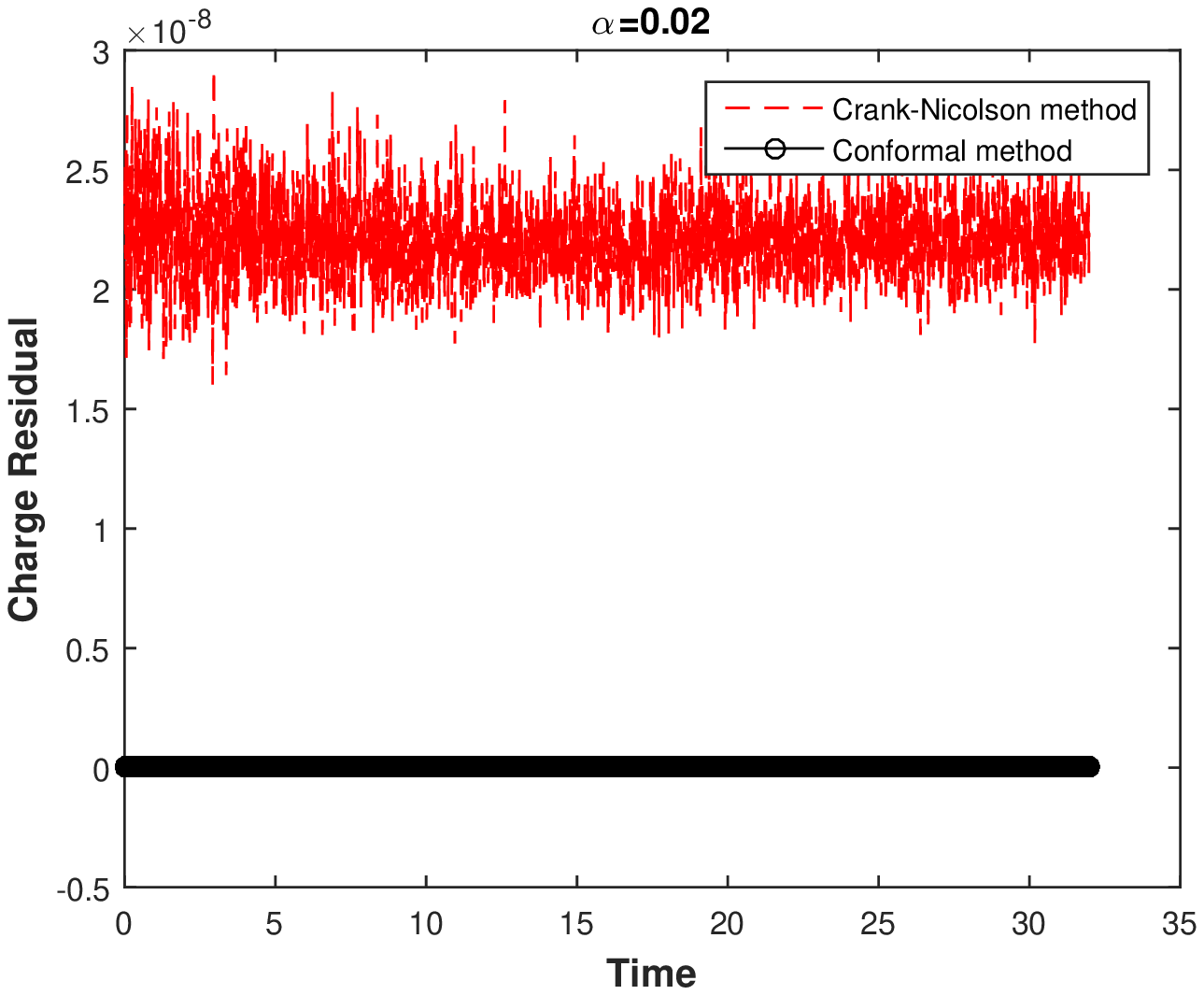}\\
  \includegraphics[height=4.0cm,width=4.5cm]{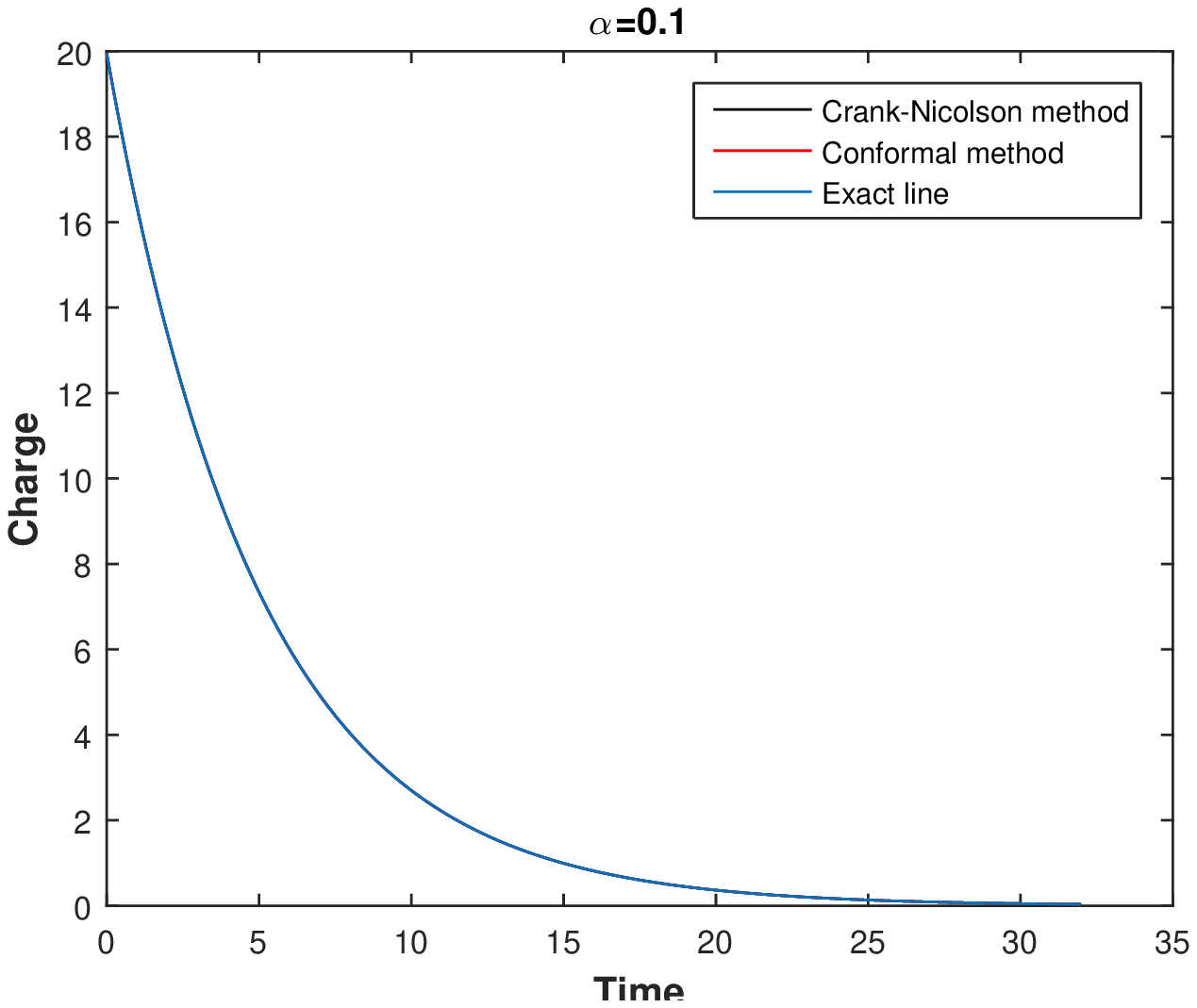}
  \includegraphics[height=4.0cm,width=4.5cm]{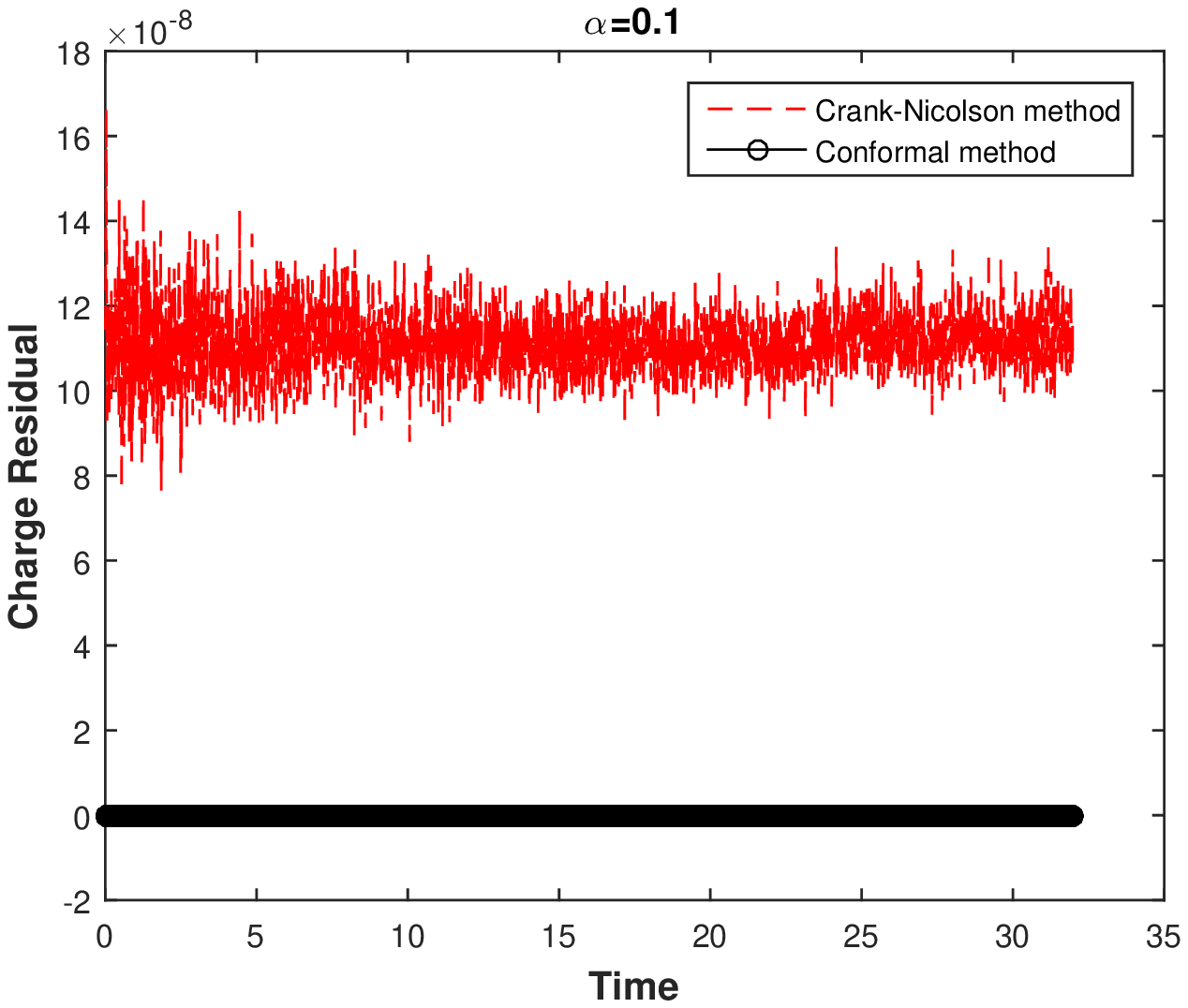}
  \caption{Evolution of the charge dissipation law averaged over 100 trajectories with $\Delta t=0.01$,~$T=32$.}\label{charge00}
  \end{center}
\end{figure}

Fig. \ref{energy_res} exhibits the discrete average energy over 100 trajectories with $\alpha=0.02$ and $\alpha=0.1$, respectively. We notice that
 for both cases the energy shows a rapid increase at the beginning, but for small $\alpha$ the energy decreases slowly while for large $\alpha$, it drops down
 severely.

\begin{figure}[th!]
\begin{center}
  \includegraphics[height=4.0cm,width=5.0cm]{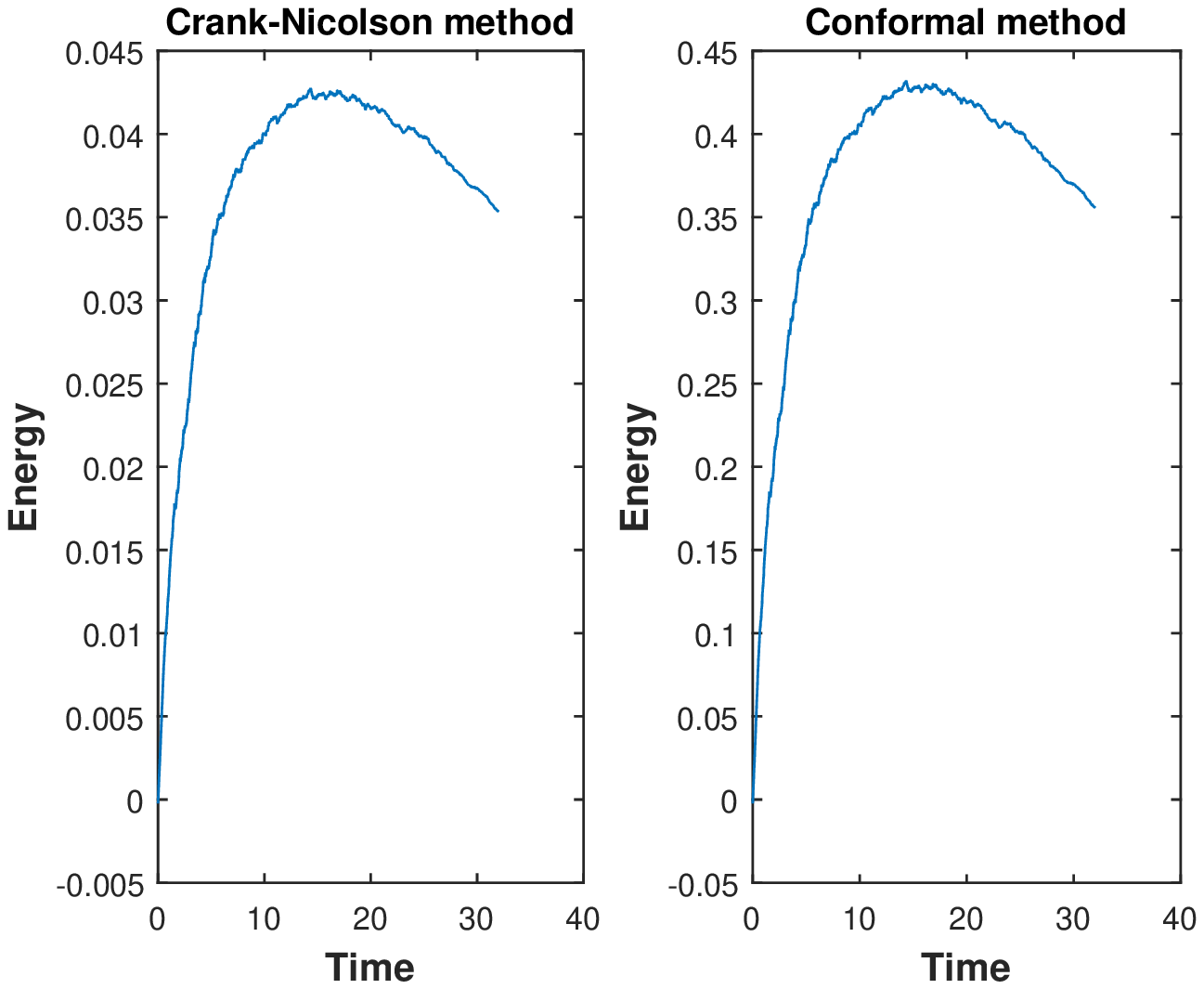}
  \includegraphics[height=4.0cm,width=5.0cm]{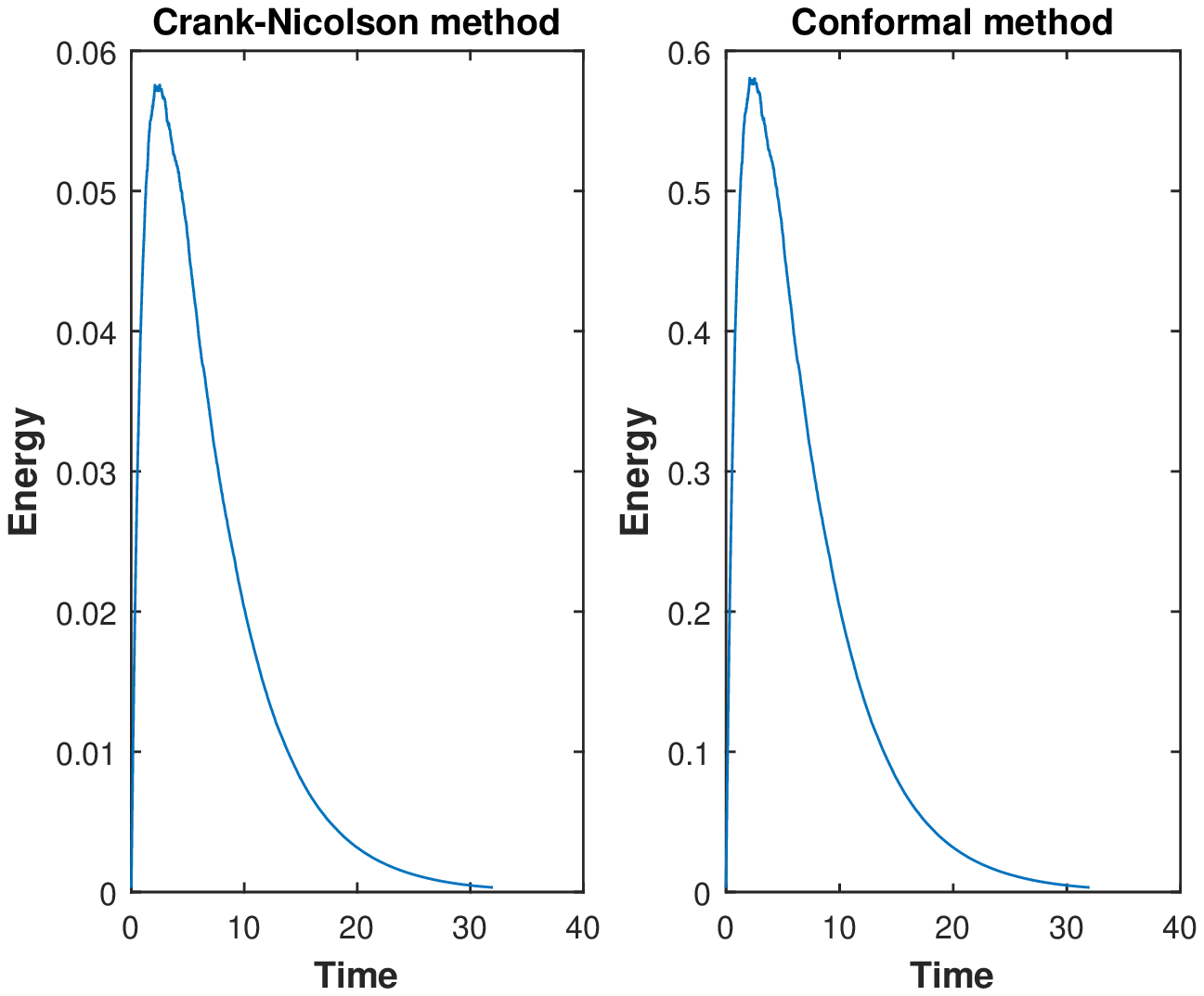}
  \caption{Evolution of the average energy over 100 trajectories with $\Delta t=0.01$, $T=32$. Left: $\alpha=0.02$; Right: $\alpha=0.1$}\label{energy_res}
  \end{center}
\end{figure}

\vspace{3mm}
\noindent{\it Example 3.}
We investigate the convergence order in temporal direction of the proposed stochastic conformal multi-symplectic method in this experiment. Define
\begin{equation*}
  e_{\Delta t}^{strong}:=\Big(\mathbb{E}\|u(\cdot,T)-u_{T}(\cdot)\|^{2}\Big)^{\frac{1}{2}},
\end{equation*}
 let $[x_{L},x_{R}]=[-1,1]$, $\Delta x=\frac{1}{256}$, $T=\frac{1}{4}$ and $u|_{t=0}=\sin(\pi x)$, and plot
 $e_{\Delta t}^{strong}$ against $\Delta t$ on a log-log scale with various combinations of $(\alpha,~\varepsilon)$ for the truncated number of Wiener process $1\leq M\leq 8$. Although we do not know the explicit form of the solution to \eqref{NLS}, we take the stochastic conformal multi-symplectic method with small time stepsize $\Delta t=2^{-14}$ as the reference solution. We then compare it to
the stochastic conformal multi-symplectic method evaluated with time steps $(2^{1}\Delta t, 2^{3}\Delta t, 2^{5}\Delta t, 2^{7}\Delta t)$ in order to estimate the rate of convergence.

We consider $\varepsilon=0$ first: Fig. \ref{order_det} shows order 2 for the $\mathbb{L}^{2}$-error $\|u(\cdot,T)-u_{T}(\cdot)\|_{\mathbb{L}^{2}}$ of the conformal multi-symplectic method for different sizes of $\alpha$.
\begin{figure}[th!]
\begin{center}
  \includegraphics[height=3.5cm,width=3.7cm]{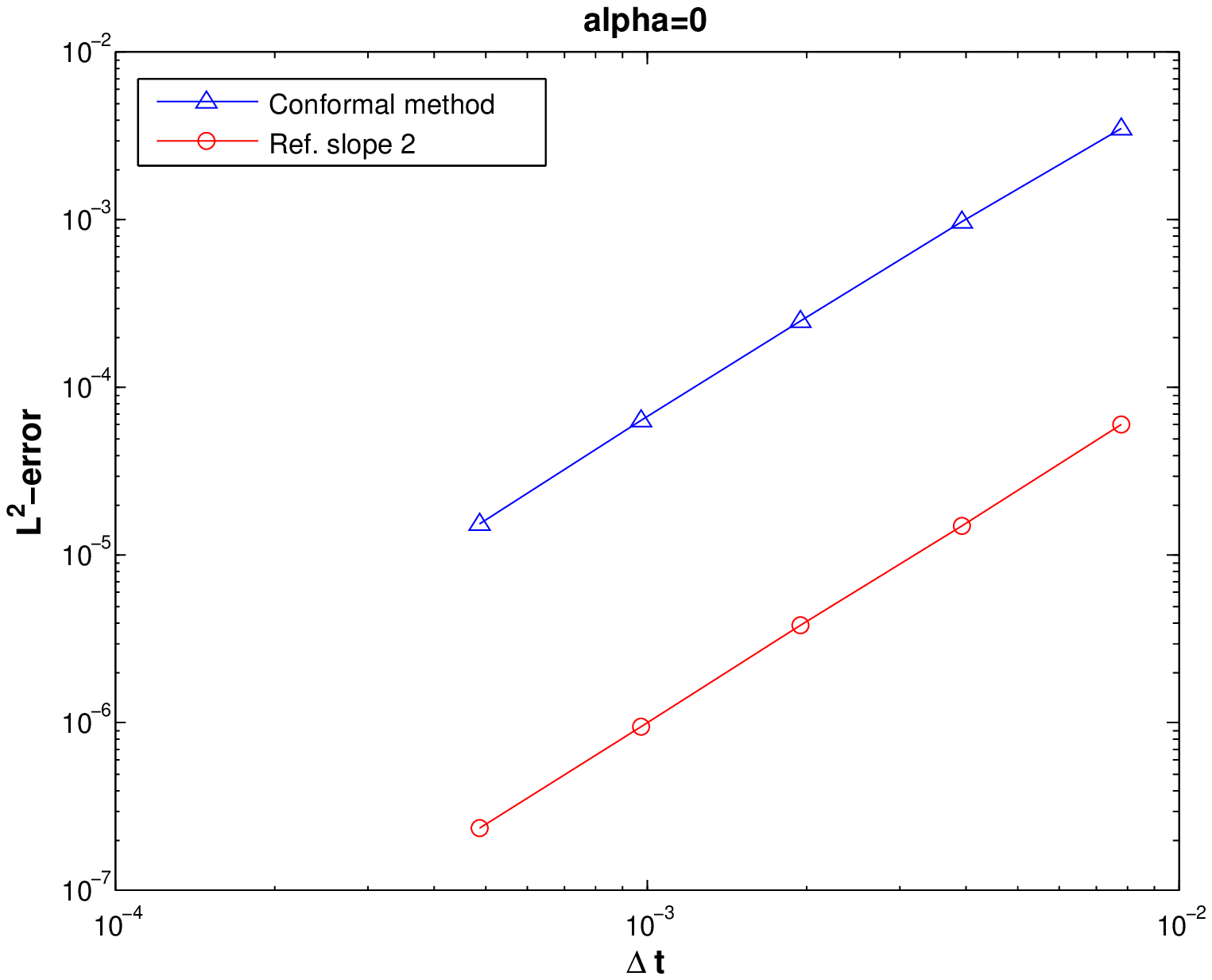}
  \includegraphics[height=3.5cm,width=3.7cm]{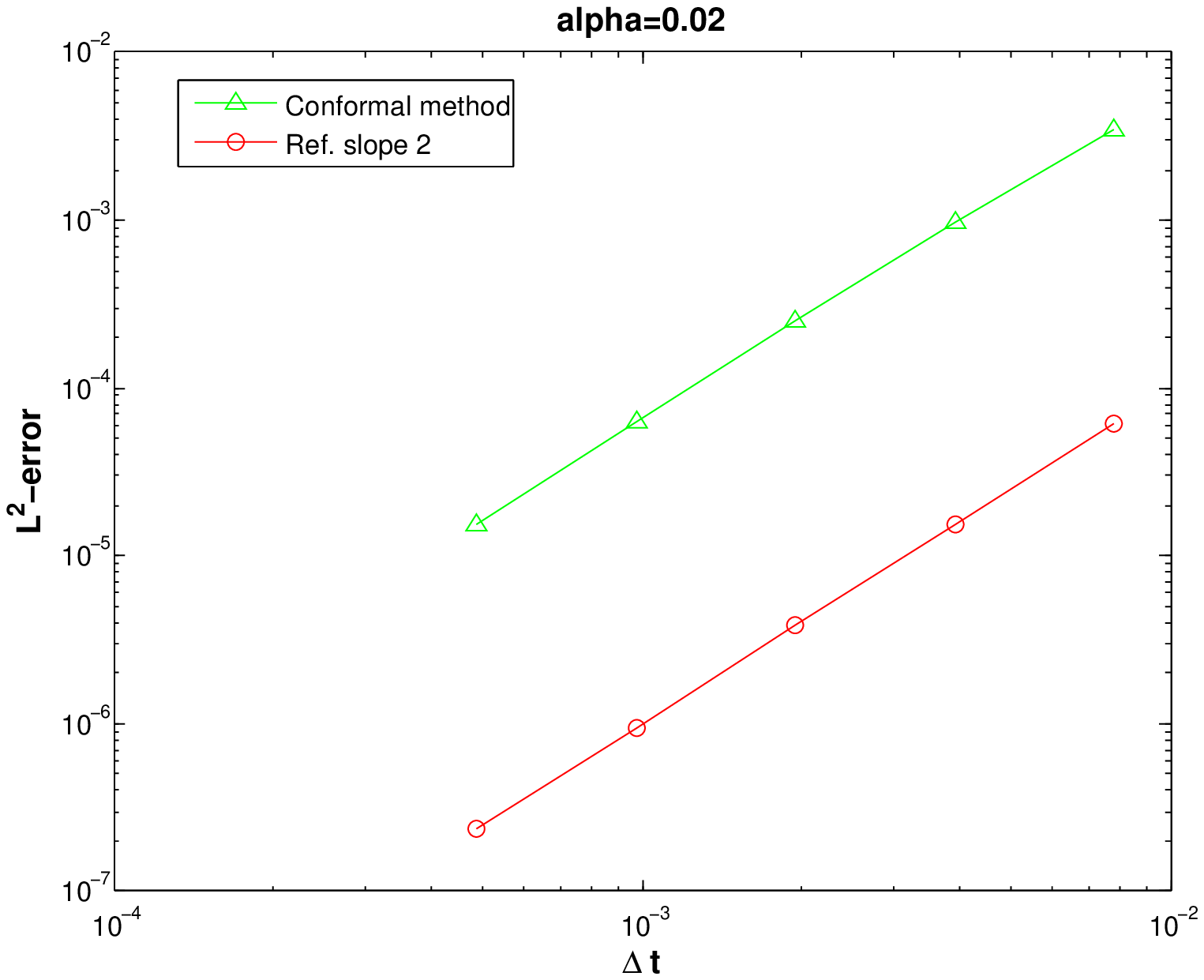}
  \includegraphics[height=3.5cm,width=3.7cm]{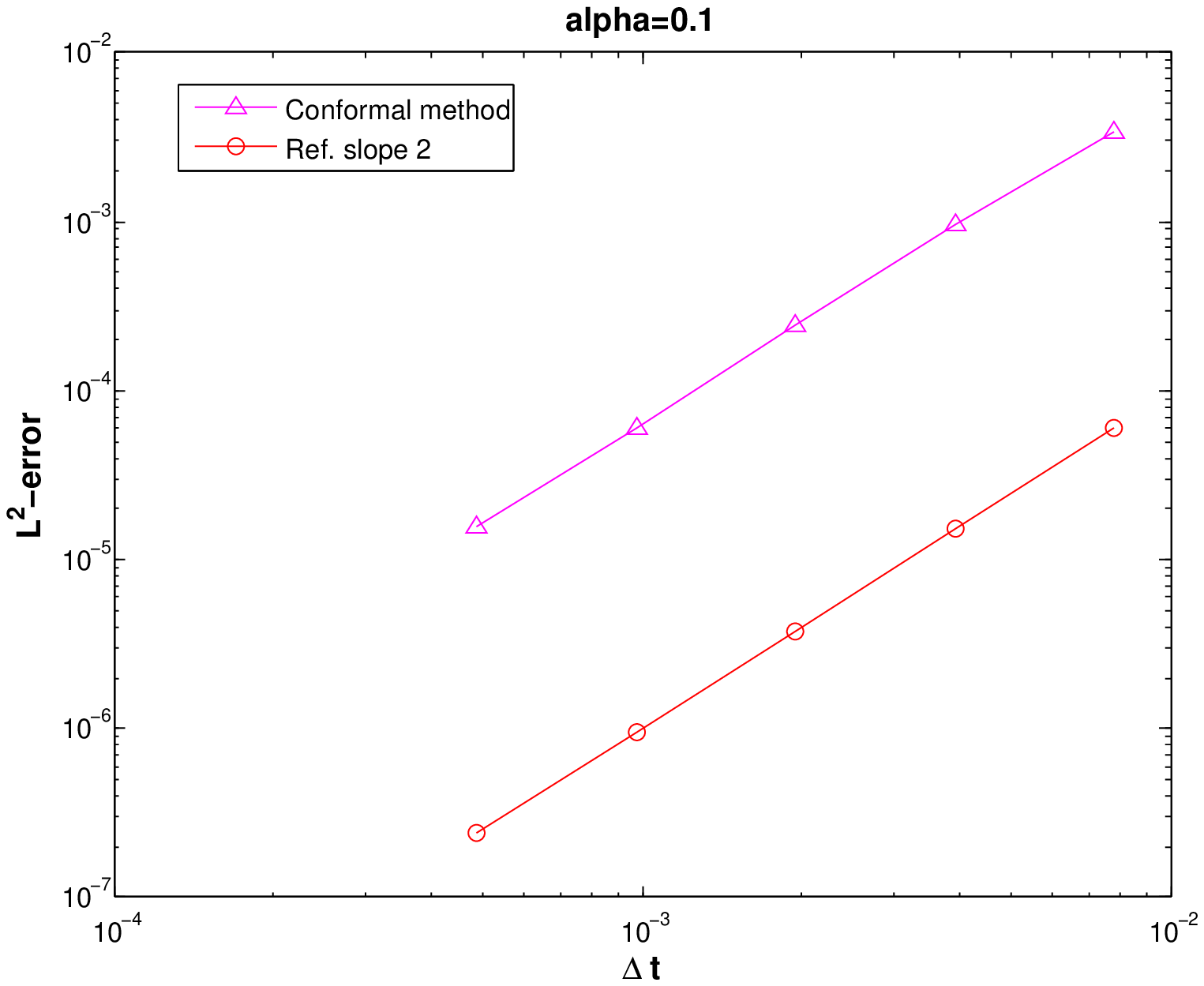}
     \caption{Rates of convergence for the deterministic case with $\alpha=0$, $\alpha=0.02$ and $\alpha=0.1$, respectively. }\label{order_det}
  \end{center}
\end{figure}

The observations are different in the stochastic case ($\varepsilon=\sqrt{2}$) where different sorts of Wiener processes depending on $M$ are used. Fig. \ref{order} presents the mean-square convergence order for the $\mathbb{L}^{2}$-error $e_{\Delta t}^{strong}$ with various sizes of $\alpha$. And 500 realizations are chosen to approximate the expectations. As is displayed in Fig. \ref{order}, the strong order of of convergence $e_{\Delta t}^{strong}$ drops from approximately 1 to 0.5 for values 1 to 8 of $M$. It is an interesting and open problem to investigate theoretically the convergence order of the proposed stochastic conformal multi-symplectic method.

\begin{figure}[th!]
\begin{center}
  \includegraphics[height=3.5cm,width=3.7cm]{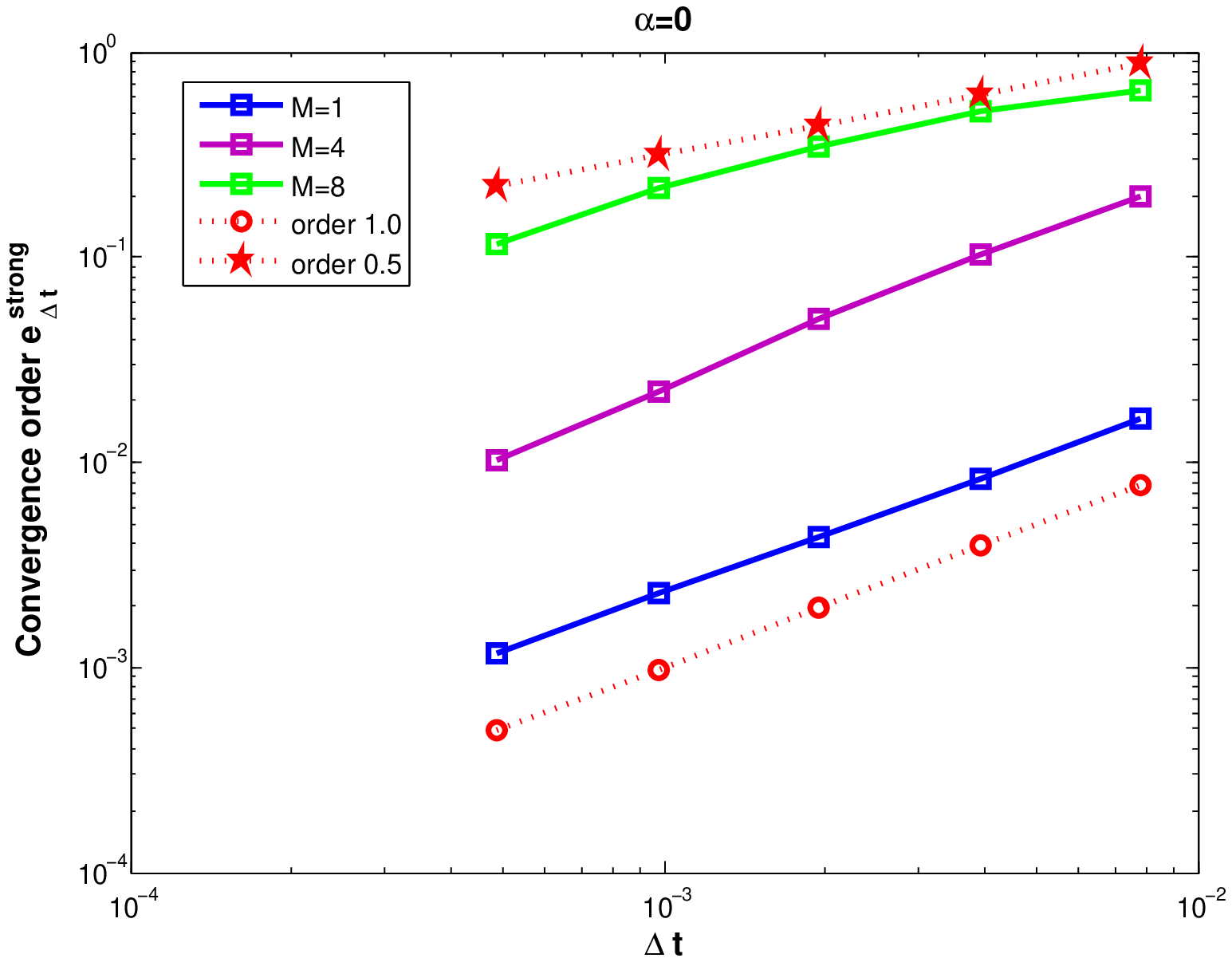}
  \includegraphics[height=3.5cm,width=3.7cm]{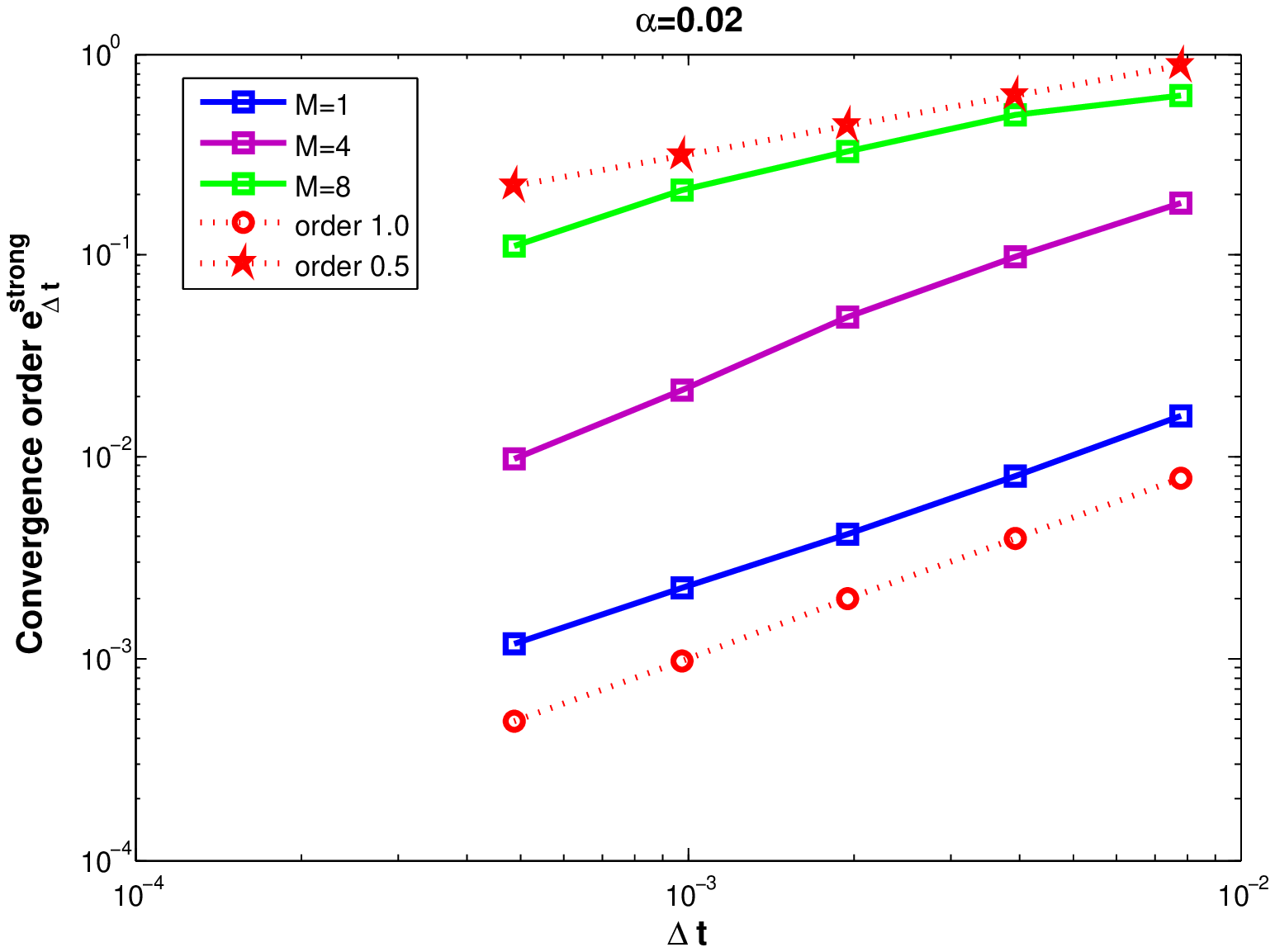}
  \includegraphics[height=3.5cm,width=3.7cm]{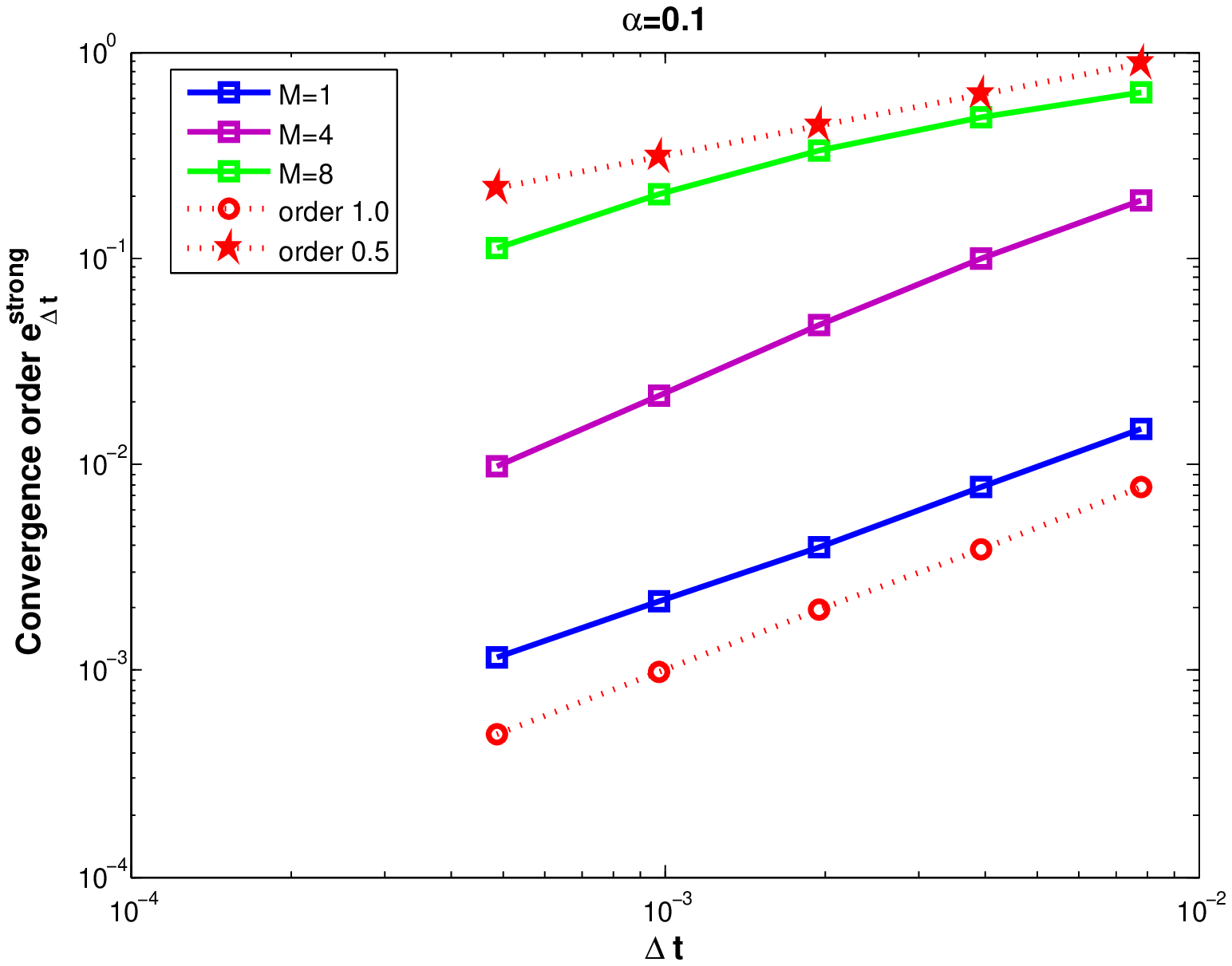}
   \caption{Mean square error versus time steps with $\alpha=0$, $\alpha=0.02$ and $\alpha=0.1$ for $1\leq M\leq 8$, respectively.}\label{order}
  \end{center}
\end{figure}

\section{Conclusions and remarks}
In this paper, we firstly investigate the intrinsic property --- stochastic conformal multi-symplectic structure --- of a class of damped stochastic Hamiltonian PDEs. It has been widely recognized that the structure-preserving methods have the remarkable superiority to conventional numerical methods when applied to Hamiltonian system, such as long-term behavior, structure-preserving, etc. In order to numerically inherit properties possessed by the damped stochastic Hamiltonian PDEs, we propose the stochastic conformal multi-symplectic numerical method, and analyze discrete versions of the corresponding properties.

Secondly, we take the damped stochastic NLS equation as an example, and present its equivalent form of damped stochastic Hamiltonian PDEs. And a stochastic conformal multi-symplectic numerical method is proposed to discretize the damped stochastic NLS equation. We show that the proposed conformal method is equivalent to apply the standard multi-symplectic numerical method to the transformed equations. It is proved that the proposed numerical method not only could preserve the discrete stochastic conformal multi-symplectic structure, but also could preserve the discrete charge exponential dissipation law exactly as the continuous problem. Moreover, the concrete relationship satisfied by the energy both for continuous problem and the numerical scheme are derived.

Finally, numerical experiments are preformed to study the good performance of the proposed stochastic conformal multi-symplectic method, compared with  a Crank-Nicolson type method when applied to discretize the damped stochastic NLS equation. It is noted that the stochastic conformal multi-symplectic method preserves the amplitudes of the wave and the dissipation rate of the charge exactly, but in contrast may alter the phases which mean the wave speeds. The discrete averaged energies all show rapid increase at the beginning and then fall down with different speed for different damping. Numerical results suggest that the stochastic conformal methods are superior to standard ones for long time simulation. At last, the mean square convergence order in temporal direction of the proposed stochastic conformal multi-symplectic method is studied numerically, which may present some instructions to investigate it theoretically and rigorously in the further work.
\appendix
\section{Proof of Theorem \ref{conlemma}}
\begin{proof}
Let $d(z_{t_{1}}^{x})^{i}$, $d(z_{t_{0}}^{x})^{i}$, $d(z_{t}^{x_{1}})^{i}$, $d(z_{t}^{x_{0}})^{i}$ denote the $i$-th components of the differential forms $dz(t_{1},x)$, $dz(t_{0},x)$, $dz(t,x_{1})$, $dz(t,x_{0})$ respectively. And let $M^{ij}$, $K^{ij}$, $D^{ij}$ be the
  elements of the matrices $M$, $K$, $D$, respectively.
  \begin{equation*}
  \begin{split}
   & \int_{x_{0}}^{x_{1}}\omega(t_{1},x)dx-\int_{x_{0}}^{x_{1}}\omega(t_{0},x)dx\\
  &  =\int_{x_{0}}^{x_{1}}\Big[\sum_{i=1}^{d}d(z_{t_{1}}^{x})^{i}\wedge\Big(\sum_{j=1}^{d}M^{ij}d(z_{t_{1}}^{x})^{j}\Big)-
    \sum_{i=1}^{d}d(z_{t_{0}}^{x})^{i}\wedge\Big(\sum_{j=1}^{d}M^{ij}d(z_{t_{0}}^{x})^{j}\Big)\Big]dx\\
    &=\int_{x_{0}}^{x_{1}}\sum_{i=1}^{d}\sum_{j=1}^{d}M^{ij}\Big(d(z_{t_{1}}^{x})^{i}\wedge d(z_{t_{1}}^{x})^{j}-d(z_{t_{0}}^{x})^{i}\wedge d(z_{t_{0}}^{x})^{j}\Big)dx\\
    &=\int_{x_{0}}^{x_{1}}\sum_{i=1}^{d}\sum_{j=1}^{d}M^{ij}\bigg[\Big(\sum_{\ell=1}^{d}
    \frac{\partial(z_{t_{1}}^{x})^{i}}{\partial(z_{t_{0}}^{x_{0}})^{\ell}}d(z_{t_{0}}^{x_{0}})^{\ell}\Big)\wedge
    \Big(\sum_{k=1}^{d}\frac{\partial(z_{t_{1}}^{x})^{j}}{\partial(z_{t_{0}}^{x_{0}})^{k}}d(z_{t_{0}}^{x_{0}})^{k}\Big)\\
    &\quad\quad-\Big(\sum_{\ell=1}^{d}
    \frac{\partial(z_{t_{0}}^{x})^{i}}{\partial(z_{t_{0}}^{x_{0}})^{\ell}}d(z_{t_{0}}^{x_{0}})^{\ell}\Big)\wedge
    \Big(\sum_{k=1}^{d}\frac{\partial(z_{t_{0}}^{x})^{j}}{\partial(z_{t_{0}}^{x_{0}})^{k}}d(z_{t_{0}}^{x_{0}})^{k}\Big)\bigg]dx\\
    &=\int_{x_{0}}^{x_{1}}\sum_{i=1}^{d}\sum_{j=1}^{d}M^{ij}\bigg[\sum_{\ell=1}^{d}\sum_{k=1}^{d}\Big(
    \frac{\partial(z_{t_{1}}^{x})^{i}}{\partial(z_{t_{0}}^{x_{0}})^{\ell}}\frac{\partial(z_{t_{1}}^{x})^{j}}{\partial(z_{t_{0}}^{x_{0}})^{k}}
    -\frac{\partial(z_{t_{0}}^{x})^{i}}{\partial(z_{t_{0}}^{x_{0}})^{\ell}}\frac{\partial(z_{t_{0}}^{x})^{j}}{\partial(z_{t_{0}}^{x_{0}})^{k}}\Big)
    d(z_{t_{0}}^{x_{0}})^{\ell}\wedge d(z_{t_{0}}^{x_{0}})^{k}\bigg]dx\\
    &=\sum_{\ell=1}^{d}\sum_{k=1}^{d}\mathcal{C}_{\ell,k}(t_{1},x_{1})d(z_{t_{0}}^{x_{0}})^{\ell}\wedge d(z_{t_{0}}^{x_{0}})^{k},
    \end{split}
  \end{equation*}
  where
  \begin{equation}\label{c}
    \mathcal{C}_{\ell,k}(t_{1},x_{1})=\sum_{i=1}^{d}\sum_{j=1}^{d}M^{ij}\int_{x_{0}}^{x_{1}}\Big(
    \frac{\partial(z_{t_{1}}^{x})^{i}}{\partial(z_{t_{0}}^{x_{0}})^{\ell}}\frac{\partial(z_{t_{1}}^{x})^{j}}{\partial(z_{t_{0}}^{x_{0}})^{k}}
    -\frac{\partial(z_{t_{0}}^{x})^{i}}{\partial(z_{t_{0}}^{x_{0}})^{\ell}}\frac{\partial(z_{t_{0}}^{x})^{j}}{\partial(z_{t_{0}}^{x_{0}})^{k}}\Big)dx.
  \end{equation}
  Similarly, we have
  \begin{align}
     \int_{t_{0}}^{t_{1}}\kappa(t,x_{1})dx-\int_{t_{0}}^{t_{1}}\kappa(t,x_{0})dx
    =\sum_{\ell=1}^{d}\sum_{k=1}^{d}\mathcal{D}_{\ell,k}(t_{1},x_{1})d(z_{t_{0}}^{x_{0}})^{\ell}\wedge d(z_{t_{0}}^{x_{0}})^{k}
  \end{align}
  with
  \begin{equation}\label{d}
        \mathcal{D}_{\ell,k}(t_{1},x_{1})=\sum_{i=1}^{d}\sum_{j=1}^{d}K^{ij}\int_{t_{0}}^{t_{1}}\Big(
    \frac{\partial(z_{t}^{x_{1}})^{i}}{\partial(z_{t_{0}}^{x_{0}})^{\ell}}\frac{\partial(z_{t}^{x_{1}})^{j}}{\partial(z_{t_{0}}^{x_{0}})^{k}}
    -\frac{\partial(z_{t}^{x_{0}})^{i}}{\partial(z_{t_{0}}^{x_{0}})^{\ell}}\frac{\partial(z_{t}^{x_{0}})^{j}}{\partial(z_{t_{0}}^{x_{0}})^{k}}\Big)dx.
  \end{equation}
  And
  \begin{equation*}
  \begin{split}
    &-\int_{x_{0}}^{x_{1}}\int_{t_{0}}^{t_{1}}a\omega(t,x)dtdx
    -\int_{x_{0}}^{x_{1}}\int_{t_{0}}^{t_{1}}b\kappa(t,x)dtdx\\
    &=-\int_{x_{0}}^{x_{1}}\int_{t_{0}}^{t_{1}}\Bigg[a\sum_{i=1}^{d}d(z_{t}^{x})^{i}\wedge \Big(\sum_{j=1}^{d}M^{ij}d(z_{t}^{x})^{j}\Big)
    +b\sum_{i=1}^{d}d(z_{t}^{x})^{i}\wedge \Big(\sum_{j=1}^{d}K^{ij}d(z_{t}^{x})^{j}\Big)\Bigg]dtdx\\
    &=-\int_{x_{0}}^{x_{1}}\int_{t_{0}}^{t_{1}}\sum_{i=1}^{d}\sum_{j=1}^{d}(aM^{ij}+bK^{ij})d(z_{t}^{x})^{i}\wedge d(z_{t}^{x})^{j}dtdx\\
    &=-\sum_{i=1}^{d}\sum_{j=1}^{d}\int_{x_{0}}^{x_{1}}\int_{t_{0}}^{t_{1}}(aM^{ij}+bK^{ij})\Big(\sum_{\ell=1}^{d}\frac{\partial (z_{t}^{x})^{i}}{\partial (z_{t_{0}}^{x_{0}})^{\ell}}d(z_{t_{0}}^{x_{0}})^{\ell}\Big)\wedge \Big(\sum_{k=1}^{d}\frac{\partial (z_{t}^{x})^{j}}{\partial (z_{t_{0}}^{x_{0}})^{k}}d(z_{t_{0}}^{x_{0}})^{k}\Big)dtdx\\
    &=-\sum_{i=1}^{d}\sum_{j=1}^{d}\mathcal{E}_{\ell,k}(t_{1},x_{1})d(z_{t_{0}}^{x_{0}})^{\ell}\wedge d(z_{t_{0}}^{x_{0}})^{k},
    \end{split}
  \end{equation*}
  where for $D=-\frac{a}{2}M-\frac{b}{2}K$,
  \begin{equation}\label{e}
            \mathcal{E}_{\ell,k}(t_{1},x_{1})=\sum_{i=1}^{d}\sum_{j=1}^{d}(-2D^{ij})\int_{x_{0}}^{x_{1}}\int_{t_{0}}^{t_{1}}
    \frac{\partial(z_{t}^{x})^{i}}{\partial(z_{t_{0}}^{x_{0}})^{\ell}}\frac{\partial(z_{t}^{x})^{j}}{\partial(z_{t_{0}}^{x_{0}})^{k}}dx.
  \end{equation}
  Equation \eqref{conformal conservation law_integral} if fulfilled if and only if
  \begin{equation}\label{eq1}
    \sum_{\ell=1}^{d}\sum_{k=1}^{d}\Big(\mathcal{C}_{\ell,k}(t_{1},x_{1})+\mathcal{D}_{\ell,k}(t_{1},x_{1})+\mathcal{E}_{\ell,k}(t_{1},x_{1})
    \Big)d(z_{t_{0}}^{x_{0}})^{\ell}\wedge d(z_{t_{0}}^{x_{0}})^{k}=0.
  \end{equation}
  Set $t_{0}$, $x_{0}$ fixed, and change the variables $t_{1}$, $x_{1}$. It's not difficult to check that, if $t_{1}$ is taken as the initial time $t_{0}$, thus $x_1=x_0$, then we have $\mathcal{C}_{\ell,k}(t_{1},x_{1})\equiv 0$, $\mathcal{D}_{\ell,k}(t_{1},x_{1})\equiv 0$ and $\mathcal{E}_{\ell,k}(t_{1},x_{1})\equiv 0$,  $\ell,k=1,\cdots,d$, because the upper and lower integral limits become the same.

  So the condition \eqref{eq1} holds, if the differential of $\mathcal{C}_{\ell,k}(t_{1},x_{1})+\mathcal{D}_{\ell,k}(t_{1},x_{1})+\mathcal{E}_{\ell,k}(t_{1},x_{1})$ with respect to $t_{1}$ can be proved to be zero, i.e.,
  \begin{equation}\label{equi}
    d_{t_{1}}\mathcal{C}_{\ell,k}(t_{1},x_{1})+d_{t_{1}}\mathcal{D}_{\ell,k}(t_{1},x_{1})+d_{t_{1}}\mathcal{E}_{\ell,k}(t_{1},x_{1})=0,\quad, \ell,k=1,\cdots,d.
  \end{equation}
  Consider the $i$-th component equation of \eqref{SFDHPDE},
  \begin{equation}
  \begin{split}
    \sum_{j=1}^{d}M^{ij}d_{t_{1}}(z_{t_{1}}^{x})^{j}+\sum_{j=1}^{d}K^{ij}\frac{\partial(z_{t_{1}}^{x})^{j}}{\partial x}dt_{1}
    &=\frac{\partial S_{1}}{\partial (z_{t_{1}}^{x})^{i}}dt_{1}+\frac{\partial S_{2}}{\partial (z_{t_{1}}^{x})^{i}}dt_{1}\circ d_{t_{1}}W\\
   & +\sum_{j=1}^{d}D^{ij}(z_{t_{1}}^{x})^{j}dt_{1}+F^{i}(t,x)dt_{1}.
  \end{split}
  \end{equation}
  Taking partial derivative with respect to $(z_{t_{0}}^{x_{0}})^{k}$, we get
  \begin{align}\label{eq2}
   & \sum_{j=1}^{d}M^{ij}d_{t_{1}}\Big(\frac{\partial(z_{t_{1}}^{x})^{j}}{\partial(z_{t_{0}}^{x_{0}})^{k}}\Big)+\sum_{j=1}^{d}K^{ij}\frac{\partial}{\partial x}\Big(\frac{\partial(z_{t_{1}}^{x})^{j}}{\partial(z_{t_{0}}^{x_{0}})^{k}}\Big)dt_{1}=\sum_{j=1}^{d}\frac{\partial^{2}S_{1}(z)}{\partial(z_{t_{1}}^{x})^{i}\partial(z_{t_{1}}^{x})^{j}}
    \frac{\partial(z_{t_{1}}^{x})^{j}}{\partial(z_{t_{0}}^{x_{0}})^{k}}dt_{1}\nonumber\\
    &\quad\quad+\sum_{j=1}^{d}\frac{\partial^{2}S_{2}(z)}{\partial(z_{t_{1}}^{x})^{i}\partial(z_{t_{1}}^{x})^{j}}
    \frac{\partial(z_{t_{1}}^{x})^{j}}{\partial(z_{t_{0}}^{x_{0}})^{k}}d_{t_{1}}W
    +\sum_{j=1}^{d}D^{ij}\frac{\partial(z_{t_{1}}^{x})^{j}}{\partial(z_{t_{0}}^{x_{0}})^{k}}dt_{1}.
  \end{align}
  Similarly, for the $j$-th component equation of \eqref{SFDHPDE}, we take partial derivative with respect to $(z_{t_{0}}^{x_{0}})^{\ell}$
 and obtain
     \begin{align}\label{eq3}
   & \sum_{i=1}^{d}M^{ji}d_{t_{1}}\Big(\frac{\partial(z_{t_{1}}^{x})^{i}}{\partial(z_{t_{0}}^{x_{0}})^{\ell}}\Big)+\sum_{i=1}^{d}K^{ji}\frac{\partial}{\partial x}\Big(\frac{\partial(z_{t_{1}}^{x})^{i}}{\partial(z_{t_{0}}^{x_{0}})^{\ell}}\Big)dt_{1}=\sum_{i=1}^{d}\frac{\partial^{2}S_{1}(z)}{\partial(z_{t_{1}}^{x})^{j}\partial(z_{t_{1}}^{x})^{i}}
    \frac{\partial(z_{t_{1}}^{x})^{i}}{\partial(z_{t_{0}}^{x_{0}})^{\ell}}dt_{1}\nonumber\\
    &\quad\quad+\sum_{i=1}^{d}\frac{\partial^{2}S_{2}(z)}{\partial(z_{t_{1}}^{x})^{j}\partial(z_{t_{1}}^{x})^{i}}
    \frac{\partial(z_{t_{1}}^{x})^{i}}{\partial(z_{t_{0}}^{x_{0}})^{\ell}}d_{t_{1}}W
    +\sum_{i=1}^{d}D^{ji}\frac{\partial(z_{t_{1}}^{x})^{i}}{\partial(z_{t_{0}}^{x_{0}})^{\ell}}dt_{1}.
  \end{align}
  Due to \eqref{c}, we get
  \begin{align}\label{eq4}
    d_{t_{1}}\mathcal{C}_{\ell,k}(t_{1},x_{1})=&-\sum_{j=1}^{d}\int_{x_{0}}^{x_{1}}\Big[\sum_{i=1}^{d}M^{ji}d_{t_{1}}\Big(
    \frac{\partial(z_{t_{1}}^{x})^{i}}{\partial(z_{t_{0}}^{x_{0}})^{\ell}}\Big)\Big]\frac{\partial(z_{t_{1}}^{x})^{j}}{\partial(z_{t_{0}}^{x_{0}})^{k}}dx\nonumber\\
    &+\sum_{j=1}^{d}\int_{x_{0}}^{x_{1}}\frac{\partial(z_{t_{1}}^{x})^{i}}{\partial(z_{t_{0}}^{x_{0}})^{\ell}}\Big[\sum_{j=1}^{d}M^{ij}d_{t_{1}}\Big(
    \frac{\partial(z_{t_{1}}^{x})^{j}}{\partial(z_{t_{0}}^{x_{0}})^{k}}\Big)\Big]dx.
  \end{align}
  Substituting \eqref{eq2} and \eqref{eq3} into \eqref{eq4}, we obtain
  \begin{align}\label{cc}
    d_{t_{1}}\mathcal{C}_{\ell,k}(t_{1},x_{1})=&-\sum_{i=1}^{d}\sum_{j=1}^{d}K^{ij}\Big(\frac{\partial(z_{t_{1}}^{x_{1}})^{i}}{\partial(z_{t_{0}}^{x_{0}})^{\ell}}
    \frac{\partial(z_{t_{1}}^{x_{1}})^{j}}{\partial(z_{t_{0}}^{x_{0}})^{k}}-\frac{\partial(z_{t_{1}}^{x_{0}})^{i}}{\partial(z_{t_{0}}^{x_{0}})^{\ell}}
    \frac{\partial(z_{t_{1}}^{x_{0}})^{j}}{\partial(z_{t_{0}}^{x_{0}})^{k}}\Big)dt_{1}\nonumber\\
    &+\sum_{i=1}^{d}\sum_{j=1}^{d}2D^{ij}\int_{x_{0}}^{x_{1}}\frac{\partial(z_{t_{1}}^{x})^{i}}{\partial(z_{t_{0}}^{x_{0}})^{\ell}}
    \frac{\partial(z_{t_{1}}^{x})^{j}}{\partial(z_{t_{0}}^{x_{0}})^{k}}dxdt_{1}.
  \end{align}
  On the other hand, according to \eqref{d} and \eqref{e}, we have
  \begin{equation}\label{dd}
    d_{t_{1}}\mathcal{D}_{\ell,k}(t_{1},x_{1})=\sum_{i=1}^{d}\sum_{j=1}^{d}K^{ij}\Big(\frac{\partial(z_{t_{1}}^{x_{1}})^{i}}{\partial(z_{t_{0}}^{x_{0}})^{\ell}}
    \frac{\partial(z_{t_{1}}^{x_{1}})^{j}}{\partial(z_{t_{0}}^{x_{0}})^{k}}-\frac{\partial(z_{t_{1}}^{x_{0}})^{i}}{\partial(z_{t_{0}}^{x_{0}})^{\ell}}
    \frac{\partial(z_{t_{1}}^{x_{0}})^{j}}{\partial(z_{t_{0}}^{x_{0}})^{k}}\Big)dt_{1}
  \end{equation}
  and
  \begin{equation}\label{ee}
    d_{t_{1}}\mathcal{E}_{\ell,k}(t_{1},x_{1})=-\sum_{i=1}^{d}\sum_{j=1}^{d}2D^{ij}\int_{x_{0}}^{x_{1}}\frac{\partial(z_{t_{1}}^{x})^{i}}{\partial(z_{t_{0}}^{x_{0}})^{\ell}}
    \frac{\partial(z_{t_{1}}^{x})^{j}}{\partial(z_{t_{0}}^{x_{0}})^{k}}dxdt_{1}.
  \end{equation}
  Then the equality \eqref{equi} results from adding \eqref{cc}, \eqref{dd} and \eqref{ee} together.

  Thus the proof is finished.
\end{proof}

\end{document}